\documentclass[11pt]{article}
\usepackage{amsmath,amssymb,amsthm,amsfonts,amstext,amsbsy,amscd}
\usepackage{hyperref}
\usepackage{mathscinet}
\usepackage{graphics}
\usepackage{graphicx}
\usepackage{enumerate}
\usepackage{mathrsfs}
\usepackage[numbers]{natbib}

\usepackage{float}
\usepackage[caption = false]{subfig}

\usepackage{float}
\usepackage{multirow}
\usepackage{multicol}
\usepackage{latexsym}

\usepackage[utf8]{inputenc}
\usepackage{dsfont}
\usepackage{color}

\usepackage[top=3.5cm, bottom=3.5cm, left=3cm, right=3cm]{geometry}

\newcommand{\F}{\mathcal{F}}

\newcommand{\Pp} {\mathcal{P}}

\newcommand{\E}{\mathbb{E}}

\newcommand{\R}{\mathbb{R}}

\newcommand{\eps}{\varepsilon}

\newcommand{\1}{\mathds{1}}

\newcommand{\var} {\ensuremath {\textnormal{Var}}}

\newcommand{\supp}{\text{supp}}
\renewcommand{\P}{\mathbb{P}}

\newtheorem{theorem}{Theorem}

\newtheorem{remark}{Remark}
\newtheorem{lemma}{Lemma}

\newtheorem{corollary}{Corollary}
\newtheorem{proposition}{Proposition}

\begin{document}
\title{A Bayesian nonparametric approach to log-concave density estimation}
\author{Ester Mariucci\footnote{Institut f\"ur Mathematik, Universität Potsdam. 
E-mail: \href{mailto:mariucci@uni-potsdam.de}{mariucci@uni-potsdam.de}.}, \ Kolyan Ray\thanks{Mathematical Institute, Leiden University. \newline E-mail: \href{mailto:k.m.ray@math.leidenuniv.nl}{k.m.ray@math.leidenuniv.nl}, \href{mailto:b.t.szabo@math.leidenuniv.nl}{b.t.szabo@math.leidenuniv.nl}.} \ and Botond Szab\'o\footnotemark[2]}

\date{}

\maketitle

\begin{abstract}
The estimation of a log-concave density on $\mathbb{R}$ is a canonical problem in the area of shape-constrained nonparametric inference. We present a Bayesian nonparametric approach to this problem based on an exponentiated Dirichlet process mixture prior and show that the posterior distribution converges to the log-concave truth at the (near-) minimax rate in Hellinger distance. Our proof proceeds by establishing a general contraction result based on the log-concave maximum likelihood estimator that prevents the need for further metric entropy calculations. We further present computationally more feasible approximations and both an empirical and hierarchical Bayes approach. All priors are illustrated numerically via simulations.
\end{abstract}

\section{Introduction}

Nonparametric shape constraints offer practitioners considerable modelling flexibility by providing infinite-dimensional families that cover a wide range of parameters whilst also including numerous common parametric families. Log-concave densities on $\R$, that is densities whose logarithm is a concave function taking values in $[-\infty,\infty)$, constitute a particularly important shape-constrained class. This class includes many well-known parametric densities that are frequently used in statistical modelling, including the Gaussian, uniform, Laplace, Gumbel, logistic, gamma distributions with shape parameter at least one, Beta$(\alpha,\beta)$ distributions with $\alpha,\beta \geq 1$ and Weibull distributions with parameter at least one.

One of the original statistical motivations for considering log-concave density estimation was the problem of estimating a unimodal density with unknown mode. While this is a natural constraint in many applications, the nonparametric MLE over this class does not exist \cite{birge1997}. Since the class of log-concave densities equals the class of strongly unimodal densities \cite{ibragimov1956}, Walther \cite{walther2002} argues that this class provides a natural alternative to the full set of all unimodal densities. The class of log-concave densities also preserves many of the attractive properties of Gaussian distributions, such as closure under convolution, marginalization, conditioning and taking products. One can therefore view log-concave densities as a natural infinite-dimensional surrogate for Gaussians that retain many of their important features yet allow substantially more freedom, such as heavier tails. For these reasons, estimation of log-concave densities has received significant attention in recent years, particularly concerning the performance of the log-concave MLE  \cite{dumbgen2009,seregin2010,cule2010,cule2010b,dumbgen2011,doss2016,kim2016,kim2016b}.

Outside density estimation, log-concavity as a modelling assumption has found applications in many statistical problems, such as mixture models \cite{walther2002,balabdaoui2017}, tail index estimation \cite{muller2009}, clustering \cite{cule2010b}, regression \cite{dumbgen2011} and independent component analysis \cite{samworth2012}. For general reviews of inference with log-concave distributions and estimation under shape constraints, see \cite{walther2010} and \cite{groeneboom2014} respectively.

The Bayesian approach provides a natural way to encode shape constraints via the prior distribution, for instance under monotonicity \cite{shively2009,khazaei2010,salomond2014,reiss2018} or convexity contraints \cite{shively2011,hannah2011,han2017}. We present here a Bayesian nonparametric method for log-concave density estimation on $\R$ based on an exponentiated Dirichlet process mixture prior, which we show converges to a log-concave truth in Hellinger distance at the (near-)minimax rate. To the best of our knowledge, this is the first Bayesian nonparametric approach to this problem. We also study two computationally motivated approximations to the full Dirichlet process mixture based on standard Dirichlet process approximations, namely the Dirichlet multinomial distribution and truncating the stick-breaking representation (see Chapter 4.3.3 of \cite{vandervaartbook2017}). We further propose both an empirical and hierarchical Bayes approach that have clear practical advantages, while behaving similarly to the above in simulations. All of these priors are easily implementable using a random walk Metropolis-Hastings within Gibbs sampling algorithm, which we illustrate in Section \ref{sec:simulation}.

An advantage of the Bayesian method is that point estimates and credible sets can be approximately computed as soon as one is able to sample from the posterior distribution. In particular, the posterior yields easy access to statements on Bayesian uncertainty quantification as we show numerically in Section \ref{sec:simulation}. Our numerical results suggest that pointwise credible sets have reasonable coverage at moderate sample sizes.

The Bayesian approach also permits inference about multiple quantities, such as functionals, in a unified way using the posterior distribution. A particular functional of interest is the mode of a log-concave density. While the pointwise limiting distribution of the log-concave MLE is known \cite{balabdaoui2009}, it depends in a complicated way on the unknown density making it difficult to use to construct a confidence interval for the mode. An alternative approach to constructing a confidence interval based on comparing the log-concave MLE with the mode constrained MLE has recently been proposed \cite{doss2016a}. For the Bayesian, the marginal posterior of the mode provides a natural approach to both estimation and uncertainty quantification. Indeed, it is easy to construct Bayesian credible intervals as we demonstrate numerically in Section \ref{sec:simulation}. Whether such an approach is theoretically justified from a frequentist perspective is a subtle question related to the semiparametric Bernstein-von Mises phenomenon (Chapter 12 of \cite{vandervaartbook2017}) that is, however, beyond the scope of this article. We also note that other constraints, such as a known mode \cite{doss2016b}, can similarly be enforced through suitable prior calibration.

Given the good performance of the log-concave MLE, one might expect that Bayesian procedures, being driven by the likelihood, behave similarly well. This is indeed the case, as we show below. Our proof relies on the classic testing approach of Ghosal et al. \cite{ghosal2000} with interesting modifications in the log-concave setting. The existence and optimality of the MLE in Hellinger distance is closely linked to a uniform control of bracketing entropy \cite{vandegeer2000}. In our setting, one can exploit the affine equivariance of the log-concave MLE (Remark 2.4 of \cite{dumbgen2011}) to circumvent the need to control the metric entropy of the whole space by reducing the problem to studying a subset satisfying restrictions on the first two moments of the underlying density. This is a substantial reduction, since obtaining sharp entropy bounds in even this reduced case is highly technical, see Theorem 4 of Kim and Samworth \cite{kim2016}. One can then use the MLE to construct suitable plug-in tests with exponentially decaying type-II errors as in Gin\'e and Nickl \cite{gine2011} that take full advantage of the extra structure of the problem compared to the standard Le Cam-Birg\'e testing theory for the Hellinger distance \cite{lecam1986}. Indeed, a naive attempt to control the entropy directly, as is standard in the Bayesian nonparametrics literature (e.g. \cite{ghosal2000}), results in an overly small set on which the prior must place most of its mass. This leads to unnecessary restrictions on the prior, which in particular are not satisfied by the priors we consider in Section \ref{sec:main results}, see Remark \ref{rem:entropy}. Beyond this, there remain significant technical hurdles to proving that the prior places sufficient mass in a Kullback-Leibler neighbourhood of the truth, in particular related to the approximation of log-concave densities using piecewise log-linear functions with suitably spaced knots.

The paper is structured as follows. In Section \ref{sec:main results} we introduce our priors and present our main results, both on general contraction for log-concave densities and for the specific priors considered here. In Section \ref{sec:simulation} we present a simulation study, including a more practical empirical Bayes implementation, with some discussion in Section \ref{sec:discussion}. In Section \ref{sec:main proofs} we present the proofs of the main results with technical results placed in Section \ref{sec:technical proofs}. Some additional simulations are found in Section \ref{sec:extra_sim}.

\textit{Notation}: For two probability densities $p$ and $q$ with respect to Lebesgue measure $\lambda$ on $\R$, we write $h^2(p,q) = \int (\sqrt{p} - \sqrt{q})^2$ for the squared Hellinger distance, $K(p,q) = \int p \log \tfrac{p}{q}$ for the Kullback-Leibler divergence and $V=\int p (\log \tfrac{p}{q})^2$. We denote by $P_{f_0}^n$ the product probability measure corresponding to the joint distribution of i.i.d random variables $X_1,\dots,X_n$ with density $f_0$ and write $P_{f_0} = P_{f_0}^1$. For a function $w$, we denote by $w_-'$ and $w_+'$ its left and right derivatives respectively, that is
$$ w_-'(x)=\lim_{s\nearrow x} w'(s)\quad \textnormal{and}\quad w_+'(x)=\lim_{s\searrow x} w'(s).$$
Let $\R^+ = [0,\infty)$ and for two real numbers $a,b$, let $a\wedge b$ and $a \vee b$ denote the minimum and maximum of $a$ and $b$ respectively. Finally, the symbols $\lesssim$ and $\gtrsim$ stand for an inequality up to a constant multiple, where the constant is universal or (at least) unimportant for our purposes.

\section{Main Results}\label{sec:main results}

Consider i.i.d. density estimation, where we observe $X_1,...,X_n \sim f_0$ with $f_0=e^{w_0}$ an unknown log-concave density to be estimated. Let $\F$ denote the class of upper semi-continuous log-concave probability densities on $\R$. For $\alpha>0$ and $\beta\in \R$, denote
\begin{align*}
\F_{\alpha,\beta} := \{ f \in \F :  f(x) \leq e^{\beta-\alpha |x|} \,\, \forall x\in \R \}.
\end{align*}
By Lemma 1 of Cule and Samworth \cite{cule2010}, for any log-concave density $f_0$ there exist constants $\alpha_{f_0}>0$ and $\beta_{f_0}\in\R$ such that $f_0(x) \leq e^{\beta_{f_0}-\alpha_{f_0}|x|}$ for all $x\in \R$. Consequently, any upper semi-continuous log-concave density $f_0$ belongs to $\F_{\alpha,\beta}$ for $0<\alpha\leq \alpha_{f_0}$ and $\beta\geq \beta_{f_0}$. 

We establish a general posterior contraction theorem for priors on log-concave densities using the general testing approach introduced in \cite{ghosal2000}, which requires the construction of suitable tests with exponentially decaying type-II errors. We construct plug-in tests based on the concentration properties of the log-concave MLE, similar to the linear estimators considered in \cite{gine2011,ray2013}. The MLE has been shown to converge to the truth at the minimax rate in Hellinger distance in Kim and Samworth \cite{kim2016} and the following theorem relies heavily on their result. 

\begin{theorem}\label{thm:general contraction}
Let $\F$ denote the set of upper semi-continuous, log-concave probability densities on $\R$ and let $\Pi_n$ be a sequence of priors supported on $\F$. Consider a sequence $\varepsilon_n \rightarrow 0$ such that $n^{-2/5}\lesssim \varepsilon_n\lesssim n^{-3/8-\rho}$ for some $\rho>0$ and suppose there exists a constant $C>0$ such that
\begin{equation}\label{eq:small ball}
\Pi_n \Big( f\in \F: \int_\R f_0 \Big( \log \frac{f_0}{f} \Big) \leq \varepsilon_n^2 , \quad \int_\R f_0 \Big( \log \frac{f_0}{f} \Big)^2 \leq \varepsilon_n^2  \Big) \geq \exp (-Cn\varepsilon_n^2).
\end{equation}
Then for sufficiently large $M$,
\begin{equation*}
\Pi_n (f\in \F: h(f,f_0) \geq M \varepsilon_n|X_1,...,X_n) \rightarrow 0
\end{equation*}
in $P_{f_0}^n$-probability as $n\to \infty$.
\end{theorem}

The upper bound $\varepsilon_n \lesssim n^{-3/8 - \rho}$ is an artefact of the proof arising from the exponential inequality for the log-concave MLE that we use to construct our tests, see Lemma \ref{lem:mle_exp_ineq}. Since our interest lies in obtaining the optimal rate $n^{-2/5}$, possibly up to logarithmic factors, it plays no further role in our results. It is typical in Bayesian nonparametrics to require metric entropy conditions, which come from piecing together tests for Hellinger balls into tests for the complements of balls, see for instance Theorem 7.1 of \cite{ghosal2000}. The lack of such a condition in Theorem \ref{thm:general contraction} is tied to the optimality and specific structure of the log-concave MLE. Using the affine equivariance of the MLE (Remark 2.4 of \cite{dumbgen2011}), one can reduce the testing problem to considering alternatives in the class $\F$ restricted to have zero mean and unit variance. Unlike the whole space $\F$, the bracketing Hellinger entropy of this latter set can be suitably controlled, thereby avoiding the need for additional entropy bounds.


\begin{remark}\label{rem:entropy}
Obtaining sharp entropy bounds for log-concave function classes is a highly technical task and such bounds are only available for certain restricted subsets. Even in the case of mean and variance restrictions (Theorem 4 of \cite{kim2016}) and compactly supported and bounded densities (Proposition 14 of \cite{kim2016b}), the proofs are lengthy and require substantial effort. To use such bounds for the classic entropy-based approach to prove posterior contraction would therefore require the prior to place most of its mass on the above types of restricted sets. For instance, the prior might be required to place all but exponentially small probability on $\F_{\alpha,\beta}$ for some given $\alpha>0$, $\beta\in \R$. Such a prior construction is undesirable in practice and in fact none of our proposed priors satisfy such a restriction.
\end{remark}

We now introduce a prior on log-concave densities based on an exponentiated Dirichlet process mixture. For any measurable function $w:\R \rightarrow \R$, define the density
\begin{align}\label{eq:density_transform}
f_w(x) = \frac{e^{w(x)}}{\int_\R e^{w(y)}dy},
\end{align}
which is well-defined if $\int_\R e^{w(y)}dy<\infty$.
Recall that any monotone non-increasing probability density on $\R^+$ has a mixture representation \cite{williamson1956}
\begin{equation*}\label{eq:f}
f(x) = \int_x^\infty \frac{1}{u} dP(u),
\end{equation*}
where $P$ is a probability measure on $\R^+$. Khazaei and Rousseau \cite{khazaei2010} and Salomond \cite{salomond2014} used the above representation to obtain a Bayesian nonparametric prior for monotone non-increasing densities. Unfortunately, such a convenient mixture representation is unavailable for log-concave densities and so the prior construction is somewhat more involved. Integrating the right-hand side of the last display, we obtain a function $w:\R^+ \rightarrow \R$ as follows:
\begin{equation*}
w(x) = \gamma_1 \int_0^{\infty} \frac{u \wedge x}{u} dP(u) - \gamma_2 x,
\end{equation*}
where $\gamma_1>0$, $\gamma_2\in \R$ and $P$ is a probability measure on $[0,\infty)$. Since its (left and right) derivative is monotone decreasing, $w$ is concave. While not every concave function can be represented in this way, any log-concave density on $[0,\infty)$ can be approximated arbitrary well in Hellinger distance by a function of the form $e^w/(\int e^w)$, where $w$ is as above with $P$ a discrete probability measure, see Proposition \ref{prop:piecewise_approx}. Translating the above thus gives a natural representation for a prior construction for log-concave densities on $\R$.

Consider therefore the following possibly $n$-dependent prior on the log-density $w:[a_n,b_n] \rightarrow \R$, where possibly $a_n \rightarrow -\infty$ and $b_n \rightarrow \infty$:
\begin{align}
W(x) = \gamma_1 \int_0^{b_n-a_n} \frac{u \wedge (x-a_n)}{u} dP(u) - \gamma_2 (x-a_n),
\label{eq:prior_mixture_rep}
\end{align}
where
\begin{itemize}
\item $P \sim DP(H\1_{[0,b_n-a_n]})$, the Dirichlet process with base measure $H\1_{[0,b_n-a_n]}=H(\R^+)\bar{H} \1_{[0,b_n-a_n]}$, where $0<H(\R^+)<\infty$, $\bar{H}$ is a probability measure on $\R^+$ and every subset $U\subset [0,b_n-a_n]$ satisfies $H(U)\gtrsim \lambda(U) /(b_n-a_n)^\eta$ for some $\eta\geq 0$, 
\item $\gamma_i \sim p_{\gamma_i}$, $i=1,2,$ where $p_{\gamma_1}$, $p_{\gamma_2}$ are probability densities on $[0,\infty)$ and $\mathbb{R}$ respectively, satisfying $p_{\gamma_i}(|x|)\gtrsim e^{-c_i x^{1/4}}$, $c_i>0$, for all $x\in[0,\infty)$ and $x\in\mathbb{R}$ respectively,
\item $\gamma_1$, $\gamma_2$, and $P$ are independent.
\end{itemize}
We denote by $\Pi_n$ the full prior induced by $f_W$, where $W$ is drawn as above. Some typical draws from the prior are plotted in Figure \ref{fig: prior}.

\begin{figure}[h!]
\begin{center}
\includegraphics[width = 0.6\textwidth]{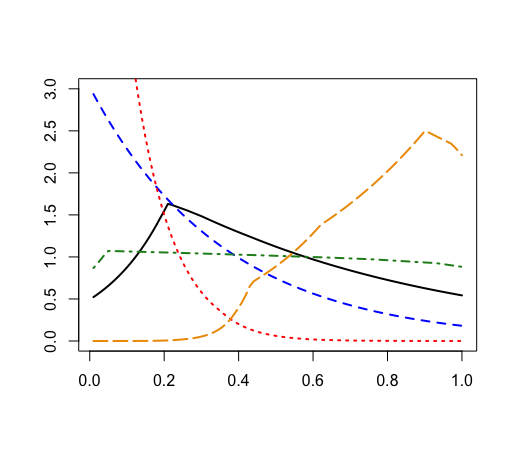} 
\caption{Prior draws with $[a_n,b_n]=[0,1]$, $\gamma_1\sim \text{Cauchy}_+(0,1)$, $\gamma_2\sim \text{Cauchy}(0,1)$, $H=U(0,1)$ and using the stick breaking construction.}
\label{fig: prior}
\end{center}
\end{figure}

\begin{remark}
If $(b_n-a_n)$ grows polynomially in $n$, then $H$ must have polynomial tails. On the other hand, if $(b_n-a_n)$ grows more slowly than any polynomial, one can relax this condition. For instance, if $H$ has a density $h$ with respect to the Lebesgue measure, then it is sufficient that $\min_{t\in[0,b_n-a_n]} h(t) \gtrsim n^{-\lambda}$ for some $\lambda >0$. In particular, if $(b_n-a_n) \lesssim \log n$, then $h$ may have exponential tails.
\end{remark}

We comment on several aspects of our prior. Firstly, since Dirichlet process draws are atomic with probability one, the prior draws \eqref{eq:prior_mixture_rep} will be piecewise linear and concave. Moreover, we could add any concave function to \eqref{eq:prior_mixture_rep}, such as an $-\gamma_3 x^2$-type term, and still have a suitable concave prior. This permits greater modelling flexibility but complicates computation. In any case, the prior described above gives optimal contraction rates and can be computed in practice, so we restrict our attention to it. Another point to note is that if $(b_n-a_n) \rightarrow \infty$ and $H$ is supported on the whole of $\R^+$, then the Dirichlet process base measure has total mass $H(\R^+) \bar{H}([a_n,b_n])\leq H(\R^+)$ for fixed $n$. This has the interpretation of assigning the prior more weight as $n \rightarrow \infty$, up to the full prior weight $H(\R^+)$. An alternative would be to re-weight the base measure to have full mass $H(\R^+)$ to give it equal weight for all $n$. This plays no role asymptotically and so we restrict to the first case for technical convenience.

A potentially more serious issue is that for fixed $n$, the support of the prior draws may not contain the support of the true density $f_0$, in which case observations outside $[a_n,b_n]$ cause the likelihood to be identically zero. While this is not a problem for $n$ large enough if $-a_n,b_n \rightarrow \infty$ fast enough (see Theorem \ref{thm: contraction}), it can be an issue for finite $n$. In practice, if one has an idea of the support of $f_0$, it is enough to select $[a_n,b_n]$ large enough to contain $\supp(f_0)$. A more pragmatic solution is to use an empirical Bayes approach and make the prior data-dependent by setting $a_n := X_{(1)}$, $b_n :=X_{(n)}$ the first and last order statistics. This ensures that the likelihood is never zero and the posterior is always well-defined. Indeed, the MLE is supported on $[X_{(1)},X_{(n)}]$ and so this can be thought of as plugging-in an estimate of the approximate support based on the likelihood. Moreover, since this approach yields the smallest support $[a_n,b_n]$ with non-zero likelihood, it also brings computational advantages. In particular, it can prevent the need to simulate the posterior distribution on potentially very large regions of $\R$ where the posterior draws are essentially indistinguishable from zero. The empirical Bayes method behaves very similarly to the prior \eqref{eq:prior_mixture_rep} in simulations and we would advocate this approach in practice. 

We first present a contraction result when the true density $f_0$ has known compact support.

\begin{theorem}\label{thm: contraction_compact}
Let $f_0\in \F_{\alpha,\beta}$ for some $\alpha>0$, $\beta\in\R$ and suppose further that $f_0$ is compactly supported. Let $a_n \equiv a$ and $b_n \equiv b$ for all $n$ and denote by $\Pi_n = \Pi$ the prior described above. If $\supp(f_0) \subset [a,b]$, then
\begin{equation*}
\Pi (f: h(f,f_0) \geq M (\log n)n^{-2/5} \mid X_1,...,X_n) \rightarrow 0
\end{equation*}
in $P_{f_0}^n$-probability for some $M=M(\alpha,\beta)>0$.
\end{theorem}

If $\supp(f_0)$ is not contained in a compact set or is unknown, it suffices to let $-a_n,b_n\rightarrow \infty$ fast enough. A slightly stronger lower bound on the tail of $p_{\gamma_1}$ is consequently required, depending on the size of $(b_n-a_n)$.

\begin{theorem}\label{thm: contraction}
Let $f_0\in \F_{\alpha,\beta}$ for some $\alpha>0$, $\beta\in\R$ and let $\Pi_n$ denote the prior described above with $-a_n,b_n\gg \log n$. Assume further that $(b_n-a_n)\lesssim n^{\mu/5}$ and that the prior density $p_{\gamma_1}$ for $\gamma_1$ satisfies the stronger lower bound $p_{\gamma_1}(x)\gtrsim e^{-c_1 x^{1/(4+\mu)}}$ for some $0 \leq \mu \leq 2$. Then
\begin{equation*}
\Pi_n (f: h(f,f_0) \geq M \varepsilon_n \mid X_1,...,X_n) \rightarrow 0
\end{equation*}
in $P_{f_0}^n$-probability for some $M=M(\alpha,\beta)>0$ and
\begin{equation}
\varepsilon_n = \max \left( (\log n) n^{-2/5} , (b_n-a_n) n^{-4/5} \right).\label{def: eps}
\end{equation}
\end{theorem}

Theorem \ref{thm: contraction_compact} follows immediately from Theorem \ref{thm: contraction} and so its proof is omitted. If $(b_n-a_n) = O((\log n)n^{2/5})$, then we obtain the minimax rate for log-concave density estimation in Theorem \ref{thm: contraction}, up to a logarithmic factor. Since the Hellinger distance dominates the total variation distance, the above also implies posterior convergence in total variation at the same rate $\varepsilon_n$ given in \eqref{def: eps}. We also note that all the above statements are proved uniformly over $f_0 \in \F_{\alpha,\beta}$.

The posterior mean, also considered in the simulation study, is not necessarily log-concave. Nevertheless one can construct log-concave density estimators by separately computing the posterior mean for each parameter $\theta,p,\gamma_1,\gamma_2$ and then constructing the corresponding log-concave density according to \eqref{eq:density_transform} and \eqref{eq:prior_mixture_rep}. Another approach is to take the smallest Hellinger ball accumulating, say, $50\%$ of the posterior mass and sample an arbitrary log-concave density from that ball. It is straightforward to verify that both of these estimators achieve the minimax concentration rate (up to the same logarithmic factor).

It is also of interest to obtain a fully Bayesian procedure that does not require the user to define the support of the prior draws.  We therefore consider a hierarchical prior where one places a prior on the end points $a$ and $b$, now not necessarily depending on $n$. This method has the advantage of employing a prior that does not depend on the data, but is slightly more involved computationally than the simple empirical Bayes approach. Assign to $(a,b)$ a prior supported on the open half-space $\{(a,b):a<b\}$ that has a Lebesgue density $\pi(a,b)$ satisfying
\begin{equation}\label{eq:hier_cond}
\pi(a,b)  \geq Ce^{-c_1|a|^q -c_2|b-a|^r} \quad \quad \text{for all }a<b
\end{equation}
and some $c_1,c_2,C,q,r>0$. Such a distribution can be easily constructed by first drawing $a\sim \pi_1$, where the Lebesgue density $\pi_1(a) \geq Ce^{-c|a|^q}$, and then independently drawing $(b-a)|a \sim \pi_2$, where $\pi_2$ is a Lebesgue density on $(0,\infty)$ satisfying $\pi_2(b-a) \geq Ce^{-c|b-a|^r}$. Conditionally on $(a_n,b_n)=(a,b)$, the prior is then exactly as above. This hierarchical construction leads to a fully Bayesian procedure that again contracts at the (near-)minimax rate.

\begin{theorem}\label{thm: hier}
Let $f_0\in \F_{\alpha,\beta}$ for some $\alpha>0$, $\beta\in\R$ and let $\Pi_n$ denote the prior described above with hyperprior $(a,b)\sim \pi(a,b)$ satisfying \eqref{eq:hier_cond}. Then
\begin{equation*}
\Pi_n (f: h(f,f_0) \geq M (\log n) n^{-2/5} \mid X_1,...,X_n) \rightarrow 0
\end{equation*}
in $P_{f_0}^n$-probability for some $M=M(\alpha,\beta)>0$.
\end{theorem}

Dirichlet process mixture priors are popular in density estimation due to the conjugacy of the posterior distribution, thereby providing methods that are highly efficient computationally. However, due to the exponentiation \eqref{eq:density_transform}, this conjugacy property no longer holds, resulting in a less attractive prior choice that brings computational challenges. In practice, it is common to use approximations of the Dirichlet process to speed up computations, see for instance Chapter 4.3.3 of \cite{vandervaartbook2017}. 

We firstly consider the Dirichlet multinomial distribution as a replacement for the Dirichlet process in our prior. By the proof of Theorem \ref{thm: contraction}, the underlying true log-concave density can be well approximated by a piecewise log-linear density with at most $N=Cn^{1/5}\log n$ knots, for some large enough constant $C>0$. In view of this, it is reasonable to take $N$ atoms in the distribution. The corresponding prior on log-concave densities then takes the form
\begin{equation}
\begin{split}
\theta_i\stackrel{iid}{\sim} \bar{H}\1_{[0,b_n-a_n]},\quad \text{for $i=1,...,N$},\\
p=(p_1,...p_{N})\sim Dir(\alpha_1,....,\alpha_N),\\
\gamma_i \stackrel{iid}{\sim} p_{\gamma_i},\quad i=1,2,\\
f_{\theta,p,\gamma_1,\gamma_2}(x)=\frac{\exp\{\gamma_1\sum_{i=1}^N \frac{\theta_i\wedge (x-a_n)}{\theta_i}p_i-\gamma_2(x-a_n) \} \1_{[a_n,b_n]}(x)}{  \int_{a_n}^{b_n}\exp\{\gamma_1\sum_{i=1}^N \frac{\theta_i\wedge (u-a_n)}{\theta_i}p_i-\gamma_2(u-a_n) \}du}, \label{prior: approx1}
\end{split}
\end{equation}
where $\alpha_i,$ $i=1,...,N,$ are chosen such that $\alpha_i=\alpha/N$ for some arbitrary $0<\alpha\leq H(\R^+)$. 

An alternative choice for the mixing prior is to truncate the stick-breaking representation of the Dirichlet process at a fixed level.  Similarly to the Dirichlet multinomial distribution, we truncate the stick-breaking process at level $N=Cn^{1/5}\log n$, resulting in the same hierarchical prior as in \eqref{prior: approx1} with the only difference being that the distribution of $p$ in the $N$-simplex is given by
\begin{align}
p_i\sim V_i \prod_{j=1}^{i-1}(1-V_j),\,\text{where $V_i\sim \text{Beta}(1,H(\R^+))$, $i=1,...,N-1$.} \label{prior: approx2}
\end{align}
Both of these computationally more efficient approximations have the same theoretical guarantees as the full exponentiated Dirichlet process prior $\Pi_n$ or its hierarchical Bayes equivalent.

\begin{corollary}\label{thm: approx}
Let $f_0\in \F_{\alpha,\beta}$ for some $\alpha>0$, $\beta\in\R$ and let $\Pi_n'$ denote either the prior \eqref{prior: approx1} or \eqref{prior: approx2}. If $-a_n,b_n\gg \log n$, $(b_n-a_n)\lesssim n^{\mu/5}$ and the prior density $p_{\gamma_1}$ for $\gamma_1$ satisfies the stronger lower bound $p_{\gamma_1}(x)\gtrsim e^{-c_1 x^{1/(4+\mu)}}$ for some $0 \leq \mu \leq 2$, then
\begin{equation*}
\Pi_n' (f: h(f,f_0) \geq M \varepsilon_n \mid X_1,...,X_n) \rightarrow 0
\end{equation*}
in $P_{f_0}^n$-probability for some $M=M(\alpha,\beta)>0$ and $\varepsilon_n$ given by \eqref{def: eps}.

If we additionally assign a hyperprior $\pi(a,b)$ satisfying \eqref{eq:hier_cond} to $(a,b)$ and no longer require the strong lower bound for $p_{\gamma_1}$, then the above holds for the resulting posterior with $\eps_n = (\log n)n^{-2/5}$.
\end{corollary}

The proofs of Theorem \ref{thm: contraction} and Corollary \ref{thm: approx} establish the small-ball probability \eqref{eq:small ball} by approximating a log-concave density in $\F_{\alpha,\beta}$ with a suitable piecewise log-linear density. This approximation requires several key properties, which make its construction non-standard and technically involved, and it may be of independent interest. The proof of Proposition \ref{prop:piecewise_approx} is deferred to Section \ref{sec:linear proofs}.

\begin{proposition}\label{prop:piecewise_approx}
Let $f_0\in \F_{\alpha,\beta}$ and $([a_n,b_n])_n$ be a sequence of compact intervals such that $[-\tfrac{8}{5\alpha} \log n,\tfrac{8}{5\alpha} \log n]\subset [a_n,b_n]$ and $(b_n - a_n) = o(n^{4/5})$. For any $n \geq n_0$, where $n_0$ is an integer depending only on $\alpha$ and $\beta$, there exists a log-concave density $\bar{f}_n$ that is piecewise log-linear with $\bar{N} \leq C(\alpha,\beta) n^{1/5} \log n$ knots $z_1,...,z_{\bar{N}}\in[0,b_n-a_n]$ satisfying the following properties:
\begin{itemize}
\item[(i)] $h^2(f_0,\bar{f}_n) \leq C(\alpha,\beta) [(\log n)^2 n^{-4/5} + (b_n-a_n)^2 n^{-8/5}]$,
\item[(ii)] $\{ x \in \R: \bar{f}_n(x) >0 \} = [a_n,b_n]$,
\item[(iii)] the knots are $cn^{-6/5}\log n$-separated for some universal constant $c>0$,
\item[(iv)] $f_0(x)\leq C(\alpha,\beta) \bar{f}_n(x)$ for all $x\in [a_n,b_n]$,
\item[(v)] there exist $\bar\gamma_1\in [0,2(b_n-a_n) n^{4/5}]$, $|\bar\gamma_2|\leq  n^{4/5}$, $\bar\gamma_3\in\mathbb{R}$ and $(\bar{p}_1,..,\bar{p}_{\bar{N}})$ satisfying $p_i\geq 0$ and $\sum_{i=1}^{\bar{N}} p_i=1$, such that
$$\bar{f}_n(x)= \exp \left( \bar{\gamma}_1\sum_{i=1}^{\bar{N}}\frac{z_i \wedge (x-a_n)}{z_i} \bar{p}_i-\bar{\gamma}_2 (x-a_n)+\bar{\gamma}_3 \right) \1_{[a_n,b_n]}(x).$$
\end{itemize}
\end{proposition}

It is relatively straightforward to establish an approximation of $f_0$ satisfying $(i)$. However, approximating $f_0$ by $\bar{f}_n$ in a Kullback-Leibler type sense, as in \eqref{eq:small ball}, necessitates control of the support of $\bar{f}_n$ via $(ii)$ and uniform control of the ratio $f_0/\bar{f}_n$ via $(iv)$. The most difficult property to establish is the polynomial separation of the points in $(iii)$. This is needed to ensure that the Dirichlet process prior simultaneously puts sufficient mass in a neighbourhood of each of the knots $z_i$, $i=1,\dots,\bar{N}$.	 Setting $[a_n,b_n] = [-\tfrac{8}{5\alpha} \log n,\tfrac{8}{5\alpha} \log n]$ yields the following corollary.

\begin{corollary}
Let $f_0\in \F_{\alpha,\beta}$. For any $n \geq n_0$, where $n_0$ is an integer depending only on $\alpha$ and $\beta$, there exists a log-concave density $\bar{f}_n$ supported on $[-\tfrac{8}{5\alpha} \log n,\tfrac{8}{5\alpha}\log n]$ that is piecewise log-linear with $O(n^{1/5} \log n)$ knots and satisfies $h^2(f_0,\bar{f}_n) \leq C(\alpha,\beta)(\log n)^2 n^{-4/5}$. Moreover, we may take the knots to be $cn^{-6/5}\log n$-separated for some universal constant $c>0$.
\end{corollary}

\section{Simulation study}\label{sec:simulation}

We present a simulation study to assess the performance of the proposed log-concave priors for density estimation. In particular, we investigate the prior based on the truncated stick breaking representation \eqref{prior: approx2}, firstly with deterministically chosen support $[a_n,b_n]$, secondly its empirical Bayes counterpart with support $[X_{(1)},X_{(n)}]$, where $X_{(1)}$ and $X_{(n)}$ denote the smallest and largest observations, respectively, and thirdly the hierarchical Bayes version where the parameters $a$ and $b-a$ are endowed with independent Cauchy and half-Cauchy distributions, respectively. In all cases we plot the posterior mean and $95\%$ pointwise credible sets, and compare them with the log-concave maximum likelihood estimator (computed using the R function ``mlelcd'').

Consider first the posterior distribution arising from the prior with deterministic support $[a_n,b_n]$. We have drawn random samples of size $n=50, 200, 500$ and $2500$ from a gamma distribution with shape and rate parameters 2 and 1, respectively. We took the number of linear pieces in the exponent of the prior to be $m=n^{1/5}\log n$, set $[a_n,b_n]=[-2.3\log n,2.3\log n]$, endowed the break-point parameters $\theta=(\theta_1,...,\theta_m)$ with independent uniform priors on $[0,b_n-a_n]$, assigned the weight parameters $p$ a stick-breaking distribution truncated at level $m$, and endowed $\gamma_1$ and  $\gamma_2$ a half Cauchy and a Cauchy distribution, respectively, with location parameter 0 and scale parameter 1. Since the posterior distribution does not have a closed-form expression, we drew approximate samples from the posterior using a random walk Metropolis-Hastings within Gibbs sampling algorithm for 10000 iterations out of which the first 5000 are discarded as burn-in period. In Figure \ref{fig: sample_size}, we have plotted the true distribution (solid red), posterior mean (solid blue), $95\%$ pointwise credible band (dashed blue) and the maximum likelihood estimator (solid green). The data is represented by a histogram on the figures.

We see that the posterior mean gives an adequate estimator for the true log-concave density with similar, if not superior, performance compared to the more jagged maximum likelihood estimator, and the $95\%$ pointwise credible bands mostly contain the true function except for points close to zero. We further investigate the frequentist coverage properties of the pointwise  Bayesian credible sets. We repeat the above experiment for the empirical Bayes procedure 100 times (each with 2000 iterations out of which half were discarded as burn in) and report the frequencies where the density $f(x)$ at given points $x\in\{0.5,1,1.5,2,2.5,3\}$ is inside of the corresponding credible interval. We consider sample sizes $n=50,200$ and $500$ and report the empirical coverage probabilities in Table \ref{tab: coverage}. One can see that in this particular example we get quite reliable uncertainty quantification, especially for larger sample sizes.  It should be noted, however, that the frequentist coverage of Bayesian credible sets is a delicate subject in nonparametric statistics, see for instance Szab\'o et al. \cite{szabo2015}, and is beyond the scope of this article.

\begin{table}[tbp]
\centering
\scalebox{0.9}{
 \setlength{\tabcolsep}{1pt}%
 \renewcommand{\arraystretch}{1.3}%
\begin{tabular}{|c|c|c|c|c|c|c|c|}
 \hline
n~\textbackslash~x&  \,\,\,\,0.5\,\,\,\,&\,\,\,\,\,\,1\,\,\,\,\,\,&\,\,\,\,1.5\,\,\,\,&\,\,\,\,\,\,2\,\,\,\,\,\,&\,\,\,\,2.5\,\,\,\,&\,\,\,\,\,\,3\,\,\,\,\,\, \\
 \hline
 \hline
50 & 0.51 & 0.81& 0.88& 0.85& 0.87& 0.93\\
 \hline
200 & 0.68& 0.97& 0.94& 0.87& 0.94& 0.97\\
 \hline
500& 0.83 & 0.96 & 0.88 & 0.87& 0.86& 0.91\\
 \hline
\end{tabular}}
\caption{Frequencies out of 100 experiments when the empirical Bayes credible set contained the true function values at points $x=0.5,1,...,3$. From top to bottom the sample size increases from $n=50$ until $n=500$.}
\label{tab: coverage}
\end{table}

\begin{figure}[h!]
\begin{center}
\includegraphics[width = 0.9\textwidth]{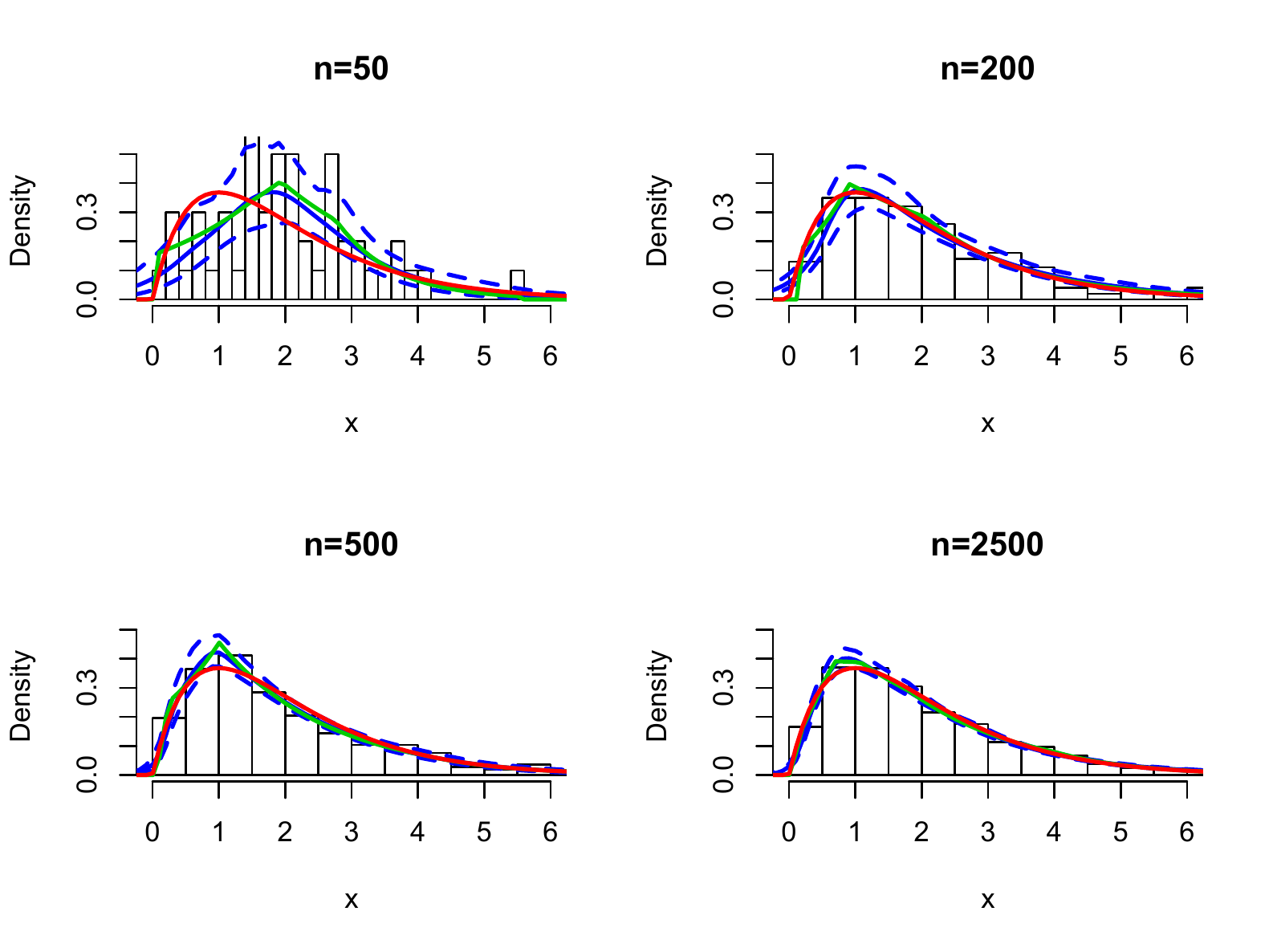} 
\caption{Prior with $[a_n,b_n]$ selected deterministically: the underlying Gamma(2,1) density function (red), posterior mean (solid blue), pointwise credible bands (dashed blue), maximum likelihood estimator (solid green) and data is represented with a histogram. We have increasing sample size from left to right and top to bottom $n=50, 200, 500$ and $2500$.}
\label{fig: sample_size}
\end{center}
\end{figure}

We next investigate the behaviour of the empirical and hierarchical Bayes versions of the proposed prior. We again simulate $n=50,200, 500$ and $2500$ independent draws from a Gamma(2,1) distribution and set the compact support of the prior densities to be $[a_n,b_n]=[X_{(1)},X_{(n)}]$, that is the smallest and largest observations, for the empirical Bayes procedure and $[a,b]$, with $a\sim Cauchy(0,1)$ and $b-a \sim Cauchy_+(0,1)$ independent, for the hierarchical Bayes method. As before we set $m=n^{1/5}\log n$ and endowed the parameters $\theta,p,\gamma_1$ and $\gamma_2$ with the same priors as above. We ran the algorithm again for 10000 iterations, taking the first half of the chain as a burn-in period. We plot the outcomes in Figures \ref{fig: sample_size_EB} and \ref{fig: sample_size_HB} for the empirical and hierarchical Bayes procedures, respectively. One can see that for $n\geq 500$ observations, the posterior mean (solid blue) closely resembles the underlying gamma density (solid red), while the fit is already reasonable for $n=200$. The pointwise $95\%$-credible bands contain the true density, even near zero, which was problematic in case of the prior with support selected deterministically. Comparing Figures \ref{fig: sample_size}, \ref{fig: sample_size_EB} and  \ref{fig: sample_size_HB}, we see that the empirical and hierarchical Bayes approaches of selecting the support $[a_n,b_n]$ in a data-driven way outperform a deterministic selection. We also note that the algorithm for the empirical Bayes method was considerably faster than the others due to the smaller support, which reduces the computation time of the normalizing constants $\int e^{w(y)}dy$ of the densities.

\begin{figure}[h!]
\begin{center}
\includegraphics[width = 0.9\textwidth]{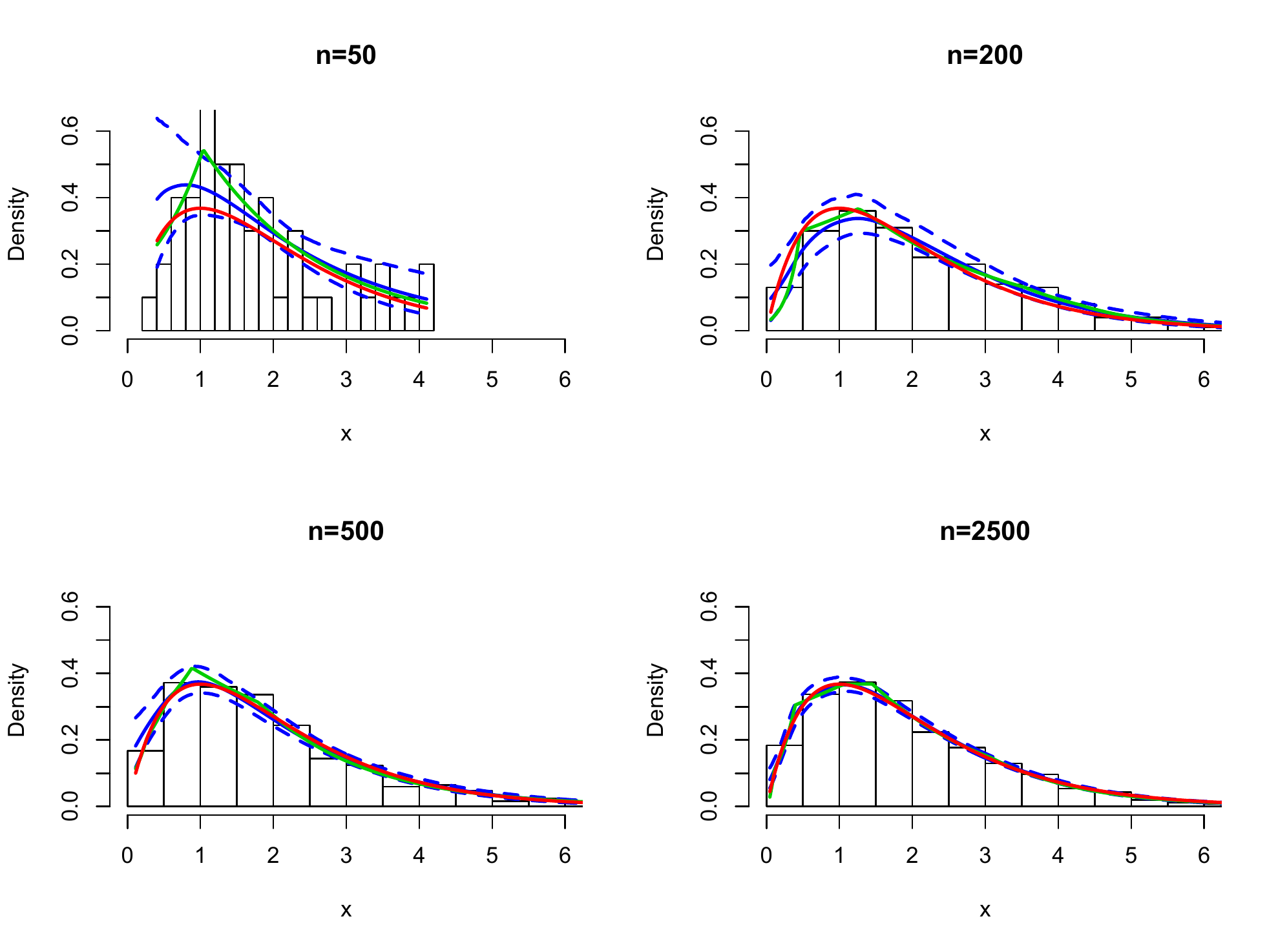} 
\caption{Empirical Bayes prior with data-driven support: the underlying Gamma(2,1) density function (red), posterior mean (solid blue), pointwise credible bands (dashed blue) and data is represented with a histogram. We have increasing sample size from left to right and top to bottom $n=50, 200, 500$ and $2500$.}
\label{fig: sample_size_EB}
\end{center}
\end{figure}

\begin{figure}[h!]
\begin{center}
\includegraphics[width = 0.9\textwidth]{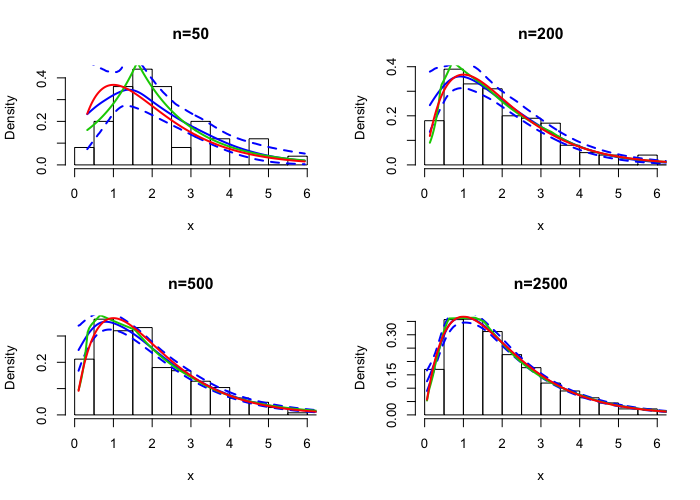} 
\caption{Hierarchical Bayes prior with data-driven support: the underlying Gamma(2,1) density function (red), posterior mean (solid blue), pointwise credible bands (dashed blue) and data is represented with a histogram. We have increasing sample size from left to right and top to bottom $n=50, 200, 500$ and $2500$.}
\label{fig: sample_size_HB}
\end{center}
\end{figure}

We then investigate the performance of the posterior distribution corresponding to the empirical and hierarchical Bayes  methods for recovering different log-concave densities and again compare them with the MLE. We have considered a standard normal distribution, a gamma distribution with shape parameter 2 and rate parameter 1, a beta distribution with shape parameters 2 and 3, and a Laplace distribution with location parameter 0 and dispersion parameter 1. In all four examples we have taken sample size $n=1500$.  The posterior mean (solid blue), the $95\%$ pointwise credible bands (dashed blue) and the MLE (green) are plotted in Figures \ref{fig: distributions_EB} and \ref{fig: distributions_HB} for the empirical and hierarchical Bayes procedures, respectively. All four subpictures for both data-driven methods show satisfactory results, both for recovery using the posterior mean and for uncertainty quantification using the pointwise credible bands. We note that the displayed plots convey typical behaviour and are representative of multiple simulations. We hence draw the conclusion that the proposed method seems to work well in practice for various choices of common log-concave densities.

\begin{figure}[h!]
\begin{center}
\includegraphics[width = 0.9\textwidth]{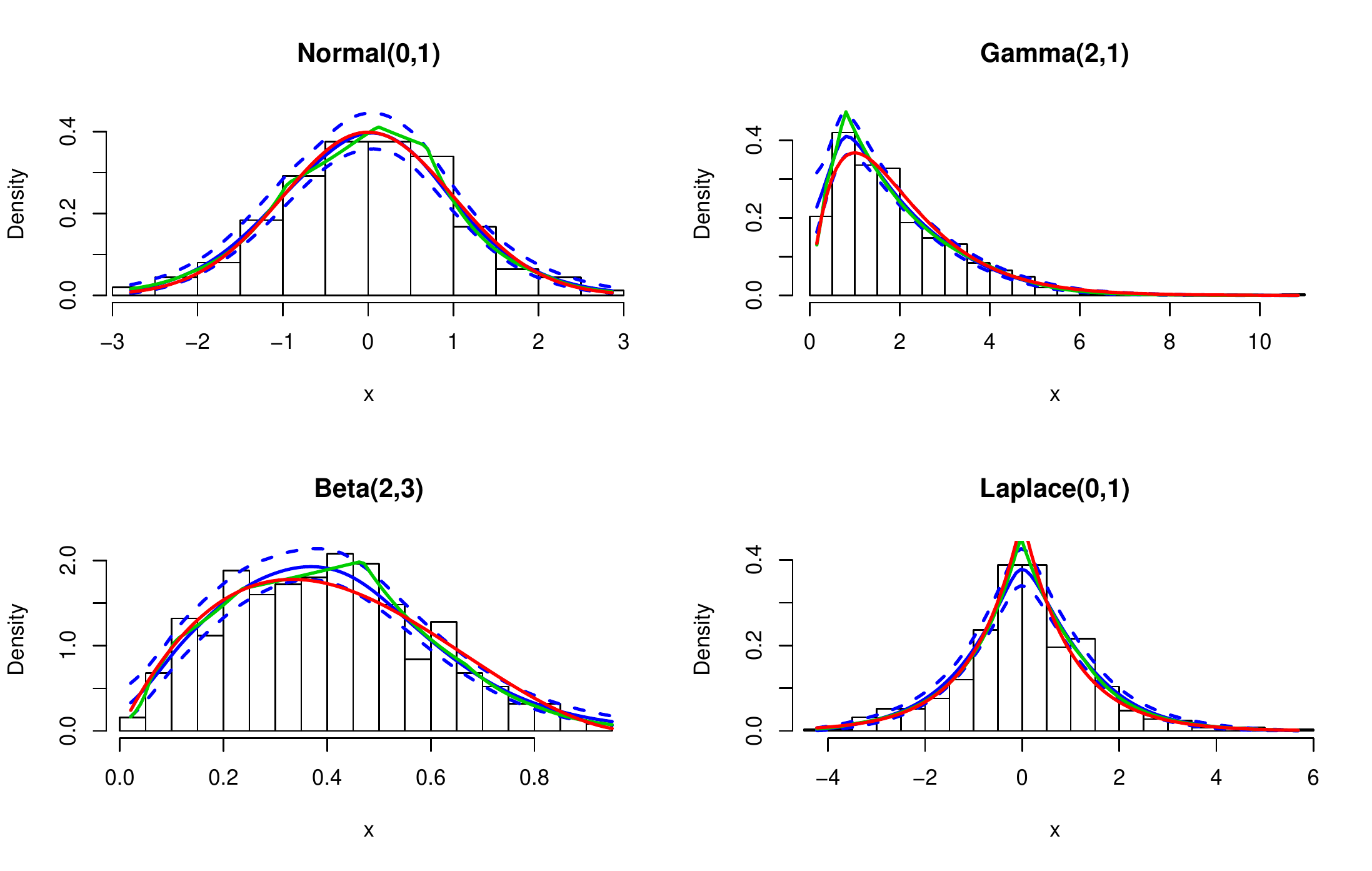} 
\caption{The underlying density function (red), empirical Bayes posterior mean (solid blue) and pointwise credible bands (dashed blue). The data is represented with a histogram. The true density functions are from left to right and top to bottom: standard Gaussian, Gamma(2,1), Beta(2,3) and Laplace(0,1).}
\label{fig: distributions_EB}
\end{center}
\end{figure}

\begin{figure}[h!]
\begin{center}
\includegraphics[width = 0.9\textwidth]{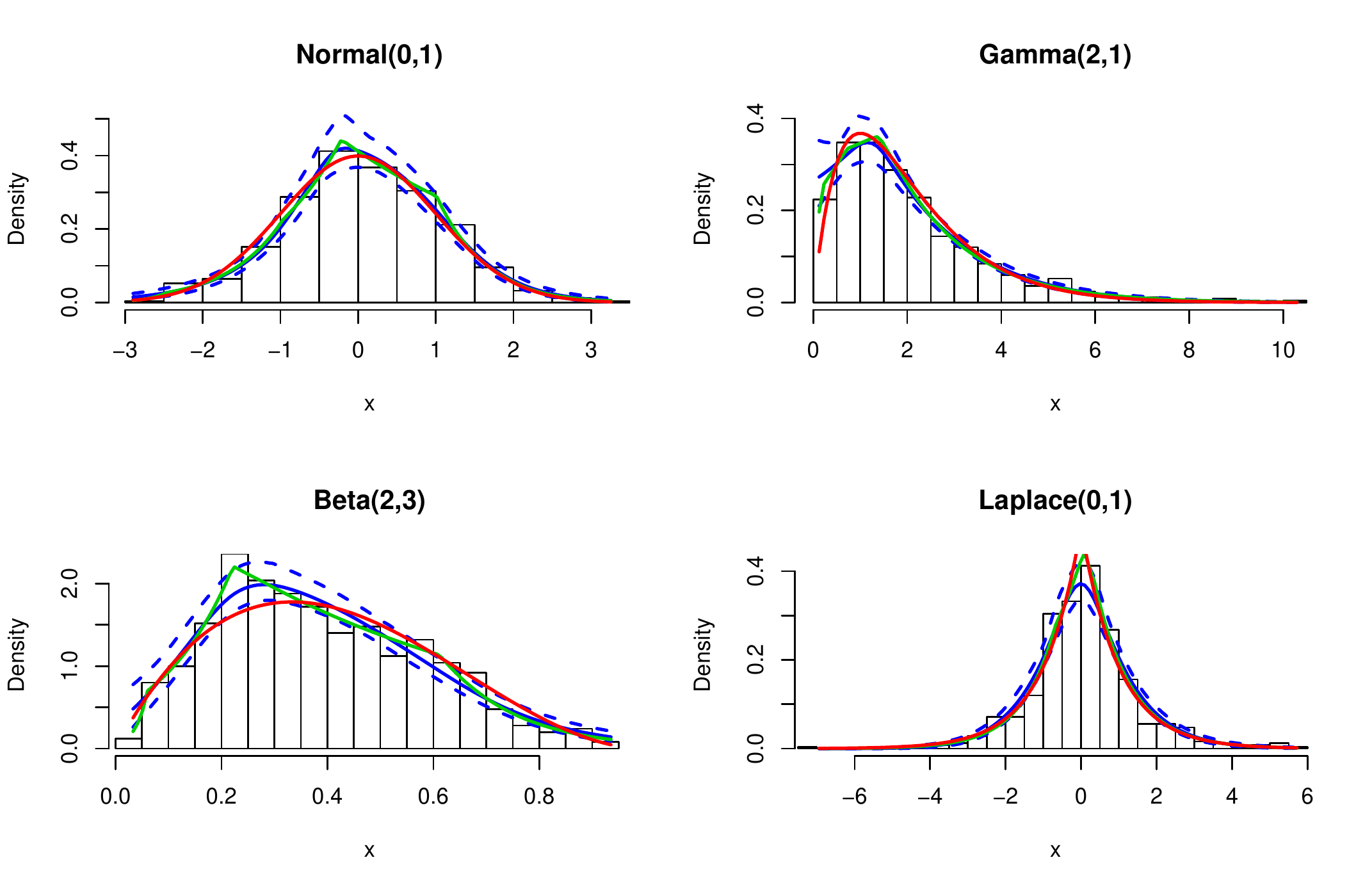} 
\caption{The underlying density function (red), hierarchical Bayes posterior mean (solid blue) and pointwise credible bands (dashed blue). The data is represented with a histogram. The true density functions are from left to right and top to bottom: standard Gaussian, Gamma(2,1), Beta(2,3) and Laplace(0,1).}
\label{fig: distributions_HB}
\end{center}
\end{figure}

Lastly, we investigate the performance of the proposed Bayesian procedures for estimating the mode of the underlying log-concave density. We consider the standard normal distribution and take i.i.d. random samples of size ranging from 50 to 20000. We run the Gibbs sampler for 20000 iterations and take the first half of the iterations as burn-in period. For each posterior draw we compute the mode and use the resulting histogram to approximate the one-dimensional marginal posterior. The histograms from the empirical Bayes procedure are displayed in Figure \ref{fig: mod_Gaussian}. One can see that the posterior concentrates around the true mode (i.e. 0) as the sample size increases. 

The marginal posterior concentrates substantially slower than $n^{-1/2}$-rate. This is as expected, since the best possible minimax rate for estimating the mode $m_0$ of a unimodal or log-concave density $f_0$ satisfying $f_0''(m_0)<0$ is $n^{-1/5}$, see \cite{Hasminskii1979,balabdaoui2009}. Indeed, the mode of the log-concave MLE attains this rate \cite{balabdaoui2009}. Interestingly, the marginal posterior does not seem to be Gaussian, which may be linked to the irregular asymptotic distribution of the mode of the log-concave MLE. This rather complicated distribution equals the mode of the second derivative of the lower invelope of a certain Gaussian process, see \cite{balabdaoui2009} for full details. A better understanding of the limiting shape of the marginal posterior would be interesting, but is beyond the scope of this article.

In the supplementary material we provide additional simulations for the marginal posterior for the mode from the empirical Bayes posterior for different underlying log-concave densities, namely the beta and gamma distributions. We also numerically investigate the applicability of our log-concave Bayesian prior for estimating mixtures of log-concave densities.

\begin{figure}[h!]
\begin{center}
\includegraphics[width = 0.9\textwidth]{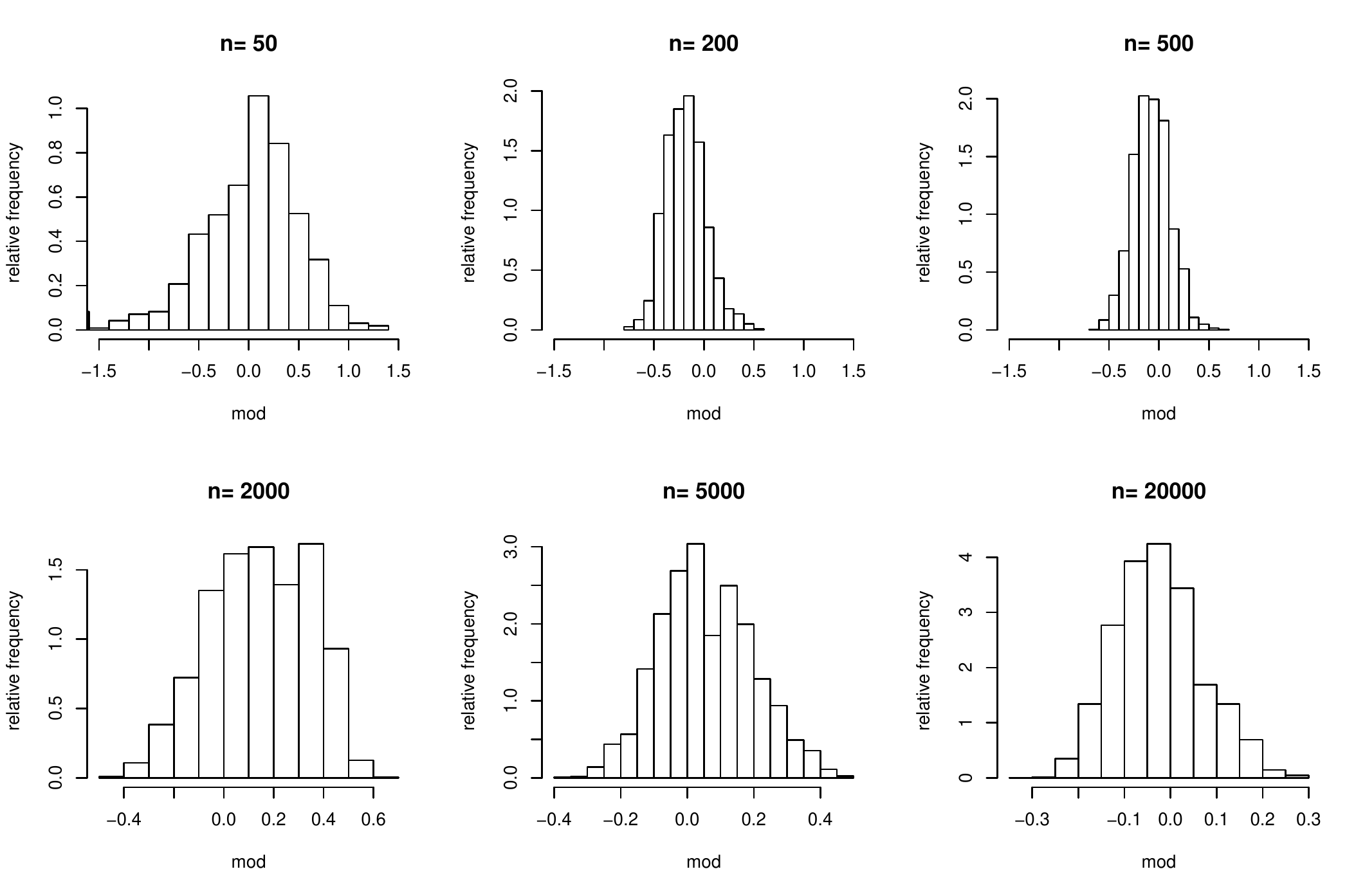} 
\caption{The empirical Bayes posterior distribution of the mode for standard normal distribution with increasing sample size from left to right and top to bottom, ranging between $n=50$ and $n=20000$.
.}
\label{fig: mod_Gaussian}
\end{center}
\end{figure}

\section{Discussion}\label{sec:discussion}
We have proposed a novel Bayesian procedure for log-concave density estimation. The prior is defined on compactly supported densities, where the support can be chosen either deterministically, empirically or using a fully Bayesian hierarchical procedure. We have shown theoretically that both the deterministic and fully Bayesian choices of the support give (near-)optimal posterior contraction rates, and have demonstrated the good small sample performance of the posterior for all three methods in a simulation study. We have also plotted the $95\%$ pointwise credible bands, which in our simulation study provide reliable frequentist uncertainty quantification. However, this might depend heavily on the choice of the underlying true density and it is unclear at present whether our methods generally provide trustworthy frequentist uncertainty quantification. The rigourous study of this question is beyond the scope of the present paper.

In our simulation study, we further investigated the behaviour of the marginal posterior for the mode functional. A natural next question is whether one can obtain semiparametric Bernstein-von Mises type results for the mode. In view of the irregular behaviour of the log-concave MLE, this is an interesting problem as it is unclear whether the limiting distribution of the posterior is indeed Gaussian.

A possible application of our proposed approach is clustering based on mixture models. Assuming that clusters have log-concave densities instead of (say) Gaussians broadens their modelling flexibility. We have executed a small simulation study to explore this direction. For simplicity we have considered a mixture of only two log-concave densities and modified our prior accordingly. In the considered examples (see the Supplementary material) our procedure performs reasonably well. However, we should note that the computational time is much worse than using simple Gaussian kernels. Extending this to mixtures with more than two (possibly unknown number of) components seems to be possible, but requires optimization of the Gibbs sampler and perhaps introducing other approximation steps, which are beyond the scope of the present paper.

Another natural question if whether one can extend these results to multivariate density estimation, especially in view of the difficulty of computing the log-concave MLE in higher dimensions. Since our present prior construction is based on using a mixture representation to model a decreasing function, which corresponds to the derivative of the concave exponent, this will require new ideas. A possible approach is presented in Hannah and Dunson \cite{hannah2011}, who place a prior over all functions that are the maximum of a set of hyperplanes. This yields a prior on the set of convex functions that could potentially be adapted to the multivariate log-concave setting.

\section{Proofs}\label{sec:main proofs}

Define the following classes of log-concave densities with mean and variance restrictions:
\begin{equation*}
\overline{\F}^{\xi,\eta} = \bigg\{ f\in \F: \mu_f := \int xf(x) dx = \xi, \quad \sigma_f^2 := \int (x-\mu_f)^2 f(x) dx = \eta \bigg\}
\end{equation*}
and
\begin{equation*}
\widetilde{\F}^{\xi,\eta} = \{ f \in \F: |\mu_f| \leq \xi , \quad |\sigma_f^2 - 1| \leq \eta \}.
\end{equation*}
Let $\hat{f}_n$ denote the log-concave MLE based on i.i.d. random variables $X_1,\dots,X_n$ arising from a density $f_0 \in \F$.

The proof of Theorem \ref{thm:general contraction} relies on a concentration inequality for the log-concave MLE based on data from moment-restricted densities. This is the content of the following lemma, whose proof is essentially contained in Kim and Samworth \cite{kim2016} for the more difficult case of general $d\geq 1$. However, we require a sharper probability bound than they provide and so make some minor modifications to their argument. The proof is deferred to Section \ref{sec:tech_remain}.

\begin{lemma}\label{lem:mle_exp_ineq}
For every $\varepsilon >0$, there exist positive constants $L_0,C,c,n_0$, depending only on $\eps$, and positive universal constants $D,d>0$, such that for all $L\geq L_0$ and $n \geq n_0$,
\begin{equation*}
\sup_{g_0 \in \overline{\F}^{0,1}} P_{g_0}^n \Big( h(\hat{g}_n,g_0) \geq Ln^{-2/5} \Big) \leq  C \exp \left( -  c n^{1/(4+2\eps)} \right) + D \exp \left( -  dL^2 n^{1/5}\right),
\end{equation*}
where $\hat{g}_n$ denotes the log-concave maximum likelihood estimator based on an i.i.d. sample $Z_1,...,Z_n$ from $g_0$.
\end{lemma}

\begin{proof}[Proof of Theorem \ref{thm:general contraction}]
As in the proof of Theorem 2.1 of \cite{ghosal2000}, using the lower bound on the small ball probability from \eqref{eq:small ball}, it suffices to construct tests $\phi_n = \phi_n (X_1,...,X_n; f_0)$ such that
\begin{equation*}
P_{f_0}^n \phi_n \rightarrow 0, \quad \quad \textnormal{ and } \quad \quad \sup_{f \in \F: h(f,f_0)\geq M \varepsilon_n} P_f^n (1-\phi_n) \leq e^{-(C+4)n\varepsilon_n^2}
\end{equation*}
for $n$ large enough, where the constant $C>0$ matches that in \eqref{eq:small ball}.

For $M_0$ a constant to be chosen below, set $\phi_n = \1 \{h(\hat{f}_n,f_0) \geq M_0 \varepsilon_n \}$, where $\hat{f}_n$ is the log-concave MLE based on i.i.d. observations $X_1,...,X_n$ from a density $f_0 \in \F$. Let $\mu_{f_0} =\E X_i$, $\sigma_{f_0}^2 = \var (X_i)$ and define $Z_i = (X_i-\mu_{f_0})/\sigma_{f_0}$, so that $\E Z_i = 0$ and $\var (Z_i)=1$. Further set $g_0(z) = \sigma_{f_0} f_0 (\sigma_{f_0}z + \mu_{f_0})$ and $\hat{g}_n(z) = \sigma_{f_0} \hat{f}_n (\sigma_{f_0}z + \mu_{f_0})$, so $g_0 \in \overline{\F}^{0,1}$. By affine equivariance (Remark 2.4 of \cite{dumbgen2011}), $\hat{g}_n$ is the log-concave maximum likelihood estimator of $g_0$ based on $Z_1,...,Z_n$.

Using the invariance of the Hellinger distance under affine transformations and Lemma \ref{lem:mle_exp_ineq} with $\eps = 1/2$, the type-I error satisfies
\begin{equation*}
P_{f_0}^n \phi_n = P_{g_0}^n (h(\hat{g}_n,g_0) \geq M_0 \varepsilon_n)\leq P_{g_0}^n (h(\hat{g}_n,g_0) \geq L_0 n^{-2/5}) \leq Ce^{-cn^{1/5}} \rightarrow 0
\end{equation*}
as $n\rightarrow \infty$ for $M_0$ large enough since $\eps_n \gtrsim n^{-2/5}$. For $f\in \F$ such that $h(f,f_0) \geq M \varepsilon_n$,
\begin{align*}
P_f^n(1-\phi_n) & = P_f^n (h(f_0,\hat{f}_n) < M_0 \varepsilon_n) \\
& \leq P_f^n (h(f_0,f) - h(f,\hat{f}_n) < M_0 \varepsilon_n)\\
& \leq P_f^n ((M-M_0)\varepsilon_n < h(f,\hat{f}_n)).
\end{align*}
Since $\eps_n \lesssim n^{-3/8-\rho}$ implies $\eps_n \lesssim n^{-3/8-\rho'}$ for any $0<\rho'\leq \rho$, we may take $\rho>0$ arbitrarily small. Applying Lemma \ref{lem:mle_exp_ineq} with $\eps(\rho)>0$ to be chosen below and $L_n= (M-M_0)\eps_n n^{2/5}$, which satisfies $L_n \geq L_0$ for $M>0$ large enough since $\eps_n \gtrsim n^{-2/5}$, yields
\begin{align*}
\sup_{f \in \F:\, h(f,f_0)\geq M \varepsilon_n} P_f^n(1-\phi_n) &\leq \sup_{g\in \overline{\F}^{0,1}} P_g^n (h(\hat{g}_n,g) > (M-M_0)\varepsilon_n)\\
&=  \sup_{g\in \overline{\F}^{0,1}} P_g^n (h(\hat{g}_n,g) > L_n n^{-2/5})\\
& \leq C(\rho)\exp \left( -c(\rho) n^\frac{1}{4+2\eps(\rho)}\right) + D \exp \left(-dL_n^2 n^{1/5} \right)
\end{align*}
for all $n \geq n_0(\rho)$. Since $\eps_n\lesssim n^{-3/8-\rho}$ by assumption, it follows that $n\eps_n^2 \lesssim  n^{1/4 -2\rho} = o(n^{1/(4+2\eps(\rho))})$ for $\eps(\rho)>0$ small enough. Therefore,
\begin{align*}
\sup_{f \in \F:\, h(f,f_0)\geq M \varepsilon_n} P_f^n(1-\phi_n) & \leq  \left( 1 + o(1) \right) D e^{-d(M-M_0)^2 n\eps_n^2}.
\end{align*}
Since we can make the constant in the exponent arbitrarily large by taking $M>0$ large enough, this completes the proof.
\end{proof}

\begin{proof}[Proof of Theorem \ref{thm: contraction}]
Let $f_0 \in \F_{\alpha,\beta}$ for some $\alpha>0$ and $\beta \in \R$. We may restrict to a suitable compactly supported density approximating $f_0$ using the first paragraph of the proof of Theorem 2 of Ghosal and van der Vaart \cite{ghosal2007}. For completeness we reproduce their argument in this paragraph. Let $\psi_n(x) = \1_{[-t_n,t_n]}(x)$ for $t_n=a' \log n$ for some $a'>\alpha^{-1}$. Define new observations $\bar{X}_1,...,\bar{X}_{\bar{n}}$ from the original observations $X_1,...,X_n$ by rejecting each $X_i$ independently with probability $1-\psi_n (X_i)$. Since $P_{f_0}[-t_n,t_n]^c \leq 2e^\beta \alpha^{-1} e^{-\alpha t_n} = o(n^{-1})$, the probability that at least one of the $X_i$'s is rejected is $o(1)$ and so the posterior based on the original and modified observations are the same with $P_{f_0}^n$-probability tending to one. Since posterior contraction is defined via convergence in $P_{f_0}^n$-probability, this implies that the posterior contraction rates are the same. The new observations come from a density $f_{0,n}$ that is proportional to $f_0 \psi_n$, which is log-concave and upper semi-continuous. Since $|1-\int f_0 \psi_n| \leq P_{f_0}[-t_n,t_n]^c = o(n^{-1})$,
\begin{equation}
\begin{split}
h^2(f_0,f_{0,n}) & \leq 2\int_\R f_0 \Bigg( 1-\frac{1}{\sqrt{\int f_0\psi_n}} \Bigg)^2 dx + \frac{2}{\int f_0 \psi_n}\int_\R f_0 (1-\sqrt{\psi_n})^2 dx \\
& \leq  2\frac{(\int f_0 \psi_n - 1)^2}{\int f_0 \psi_n} + \frac{2}{\int f_0 \psi_n}\int_{\R\backslash [-t_n,t_n]} f_0  dx  = o(n^{-1}).
\end{split}
\label{eq:log concave compact support approx}
\end{equation}
It therefore suffices to establish contraction for the posterior based on the new observations about the density $f_{0,n} = f_0\psi_n/\int f_0\psi_n$.

Under the assumed conditions on $(b_n-a_n)$, $\eps_n$ given by \eqref{def: eps} satisfies $n^{-2/5} \lesssim \eps_n \lesssim n^{-3/8-\rho}$ for some $\rho>0$ small enough. We thus apply Theorem \ref{thm:general contraction} so that we need only show the small-ball probability \eqref{eq:small ball}. Note that $f_{0,n}(x) \leq e^{\beta-\alpha|x|}(1+o(n^{-1}))$, so that $f_{0,n} \in \F_{\alpha,2\beta}$ for $n$ large enough. Since $-a_n,b_n \gg \log n$ and $b_n-a_n = o(n^{4/5})$, we may construct an approximation $\bar{f}_n$ of $f_{0,n}$ based on the interval $[a_n,b_n]$ for $n$ large enough using Proposition \ref{prop:piecewise_approx}. By Lemma 8 of \cite{ghosal2007},
\begin{align*}
\int_\R f_{0,n} \Big( \log \frac{f_{0,n}}{f_W} \Big)^k & \lesssim \big(h^2 (f_{0,n},\bar{f}_n)+h^2(\bar{f}_n, f_W) \big) \\
& \quad \times \left(1 + \log \left\| \frac{f_{0,n}}{\bar{f}_n}\right\|_{L^\infty([a_n,b_n])} + \log \left\| \frac{\bar{f}_n}{f_W}\right\|_{L^\infty([a_n,b_n])} \right)^k
\end{align*}
for $k=1,2$. By Proposition \ref{prop:piecewise_approx}(i) and (iv), the first term in the first bracket and the second term in the second bracket are $O((\log n)^2 n^{-4/5} + (b_n-a_n)^2 n^{-8/5})$ and $O(1)$ respectively.

By Proposition \ref{prop:piecewise_approx}(v), $\bar{f}_n$ has representation
\begin{align}
\bar{f}_n(x)= \exp \left( \bar{\gamma}_1\sum_{i=1}^{\bar{N}}\frac{z_i \wedge (x-a_n)}{z_i} \bar{p}_i-\bar{\gamma}_2 (x-a_n)+\bar{\gamma}_3 \right) \1_{[a_n,b_n]}(x) ,\label{eq:approx_mixture_rep}
\end{align}
where $(z_i)_{i=1}^{\bar{N}} \subset [0,b_n-a_n]$ are the knots written in increasing order, $\bar{N} = \bar{N}_n =O(n^{1/5}\log n)$ and $\sum_{i=1}^{\bar{N}}\bar{p}_i = 1$. Let $\bar{w}_n (x)= (\log \bar{f}_n (x)- \bar{\gamma}_3)\1_{[a_n,b_n]}(x) - \infty \1_{\R \backslash [a_n,b_n]}(x)$ so that $\bar{f}_n = f_{\bar{w}_n}$ using the transformation \eqref{eq:density_transform}. We may thus without loss of generality take $\bar{\gamma}_3=0$ since it is contained in the normalization \eqref{eq:density_transform}.

Suppose that $f_w$ is a (log-concave) density with support equal to $[a_n,b_n]$ and such that $\| \bar{w}_n - w\|_{L^\infty([a_n,b_n])} \leq c\varepsilon_n$. Since $\int e^w = e^{O(\varepsilon_n)} \int e^{\bar{w}_n}$, it follows that for $x\in [a_n,b_n]$, $\bar{f}_n(x)/f_w(x) \leq e^{O(\varepsilon_n)} =e^{o(1)}$. Since $h^2(\bar{f}_n,f_w)\lesssim \eps_n^2$ by Lemma 3.1 of \cite{vandervaart2008}, we can conclude that
\begin{align}
\{ w: \| \bar{w}_n - w\|_{L^\infty([a_n,b_n])} \leq c \varepsilon_n \} \subset \{ w: K(f_{0,n},f_w)\leq \varepsilon_n^2, \ V(f_{0,n},f_w) \leq \varepsilon_n^2 \}\label{eq: KL_to_L_inf}
\end{align}
for some $c>0$. It therefore suffices to lower bound the prior probability of the left-hand set.

Fix $\delta>0$ to be chosen sufficiently large below. Since the $z_i$ are $n^{-6/5}$-separated by Proposition \ref{prop:piecewise_approx}(iii), we can find a collection of disjoint intervals $(U_i)_{i=1}^{\bar{N}}$ in $[a_n,b_n]$ with Lebesgue measure $\lambda(U_i) = \varepsilon_n^\delta$ and such that $z_i \in U_i$ for $i=1,...,\bar{N}$. Further denote $U_0 := \R\backslash \cup_{i=1}^{M_n} U_i$. Let $W$ be a prior draw of the form \eqref{eq:prior_mixture_rep} with parameters $\gamma_1,\gamma_2$ and $P$.  Writing $p_i = P(U_i)$, $\bar{p}_0=0$ and using the triangle inequality, for any $x\in[a_n,b_n]$,
\begin{align}
|\bar{w}_n(x) - W(x)| & = \left| \bar{\gamma}_1 \sum_{i=1}^{\bar{N}} \frac{z_i \wedge (x-a_n)}{z_i} \bar{p}_i - \bar{\gamma}_2 (x-a_n) -  \gamma_1\int_0^\infty \frac{\theta \wedge (x-a_n)}{\theta}dP(\theta) - \gamma_2(x-a_n) \right| \nonumber\\
& \leq |\bar{\gamma}_1-\gamma_1| \int_0^\infty \frac{\theta \wedge (x-a_n)}{\theta}dP(\theta)+ \bar{\gamma}_1 \left| \int_{U_0} \frac{\theta\wedge (x-a_n)}{\theta}dP(\theta) \right| \nonumber\\
& \quad \quad+ \bar{\gamma}_1 \left| \sum_{i=1}^{\bar{N}} \int_{U_i} \frac{\theta \wedge (x-a_n)}{\theta}dP(\theta) - \sum_{i=1}^{\bar{N}} \frac{z_i \wedge (x-a_n)}{z_i} p_i  \right|  \nonumber\\
& \quad \quad + \bar{\gamma}_1 \left| \sum_{i=1}^{\bar{N}}\frac{z_i \wedge (x-a_n)}{z_i} (p_i -\bar{p}_i)  \right|  + (b_n-a_n)|\bar{\gamma}_2-\gamma_2|  \nonumber\\
& \leq |\bar{\gamma}_1-\gamma_1| + \bar{\gamma}_1 \sum_{i=1}^{\bar{N}} \sup_{\theta\in U_i} \frac{|\theta-z_i|}{\theta \wedge z_i} p_i + \bar{\gamma}_1 p_0 + \bar{\gamma}_1 \sum_{i=1}^{\bar{N}} |p_i - \bar{p}_i| + (b_n-a_n)|\bar{\gamma}_2-\gamma_2| \nonumber \\
& \leq |\bar{\gamma}_1-\gamma_1| + 2\bar{\gamma}_1 \sum_{i=1}^{\bar{N}} \frac{\lambda (U_i)}{ z_i} p_i +  \bar{\gamma}_1 \sum_{i=0}^{\bar{N}}|p_i - \bar{p}_i| + (b_n-a_n)|\bar{\gamma}_2-\gamma_2|,\label{eq: Linf_smallball}
\end{align}
where we have used in the second to last line that the maximal distance between the (piecewise) lines $(y\wedge a)/a$ and $(y\wedge b)/b$ occurs at $y=a\wedge b$ and in the last line that $\theta>z_i/2$ for all $\theta\in U_i$ for a sufficiently large choice of the parameter $\delta>0$. By Proposition \ref{prop:piecewise_approx}(v), we have $\bar\gamma_1 \leq 2 n^{4/5} (b_n-a_n)\lesssim n^{(4+\mu)/5}$. Furthermore, by the separation of the knots, $z_i \geq z_1 \geq  cn^{-6/5}$, $i=1,...,\bar{N}$, and so by the assumptions on the $(U_i)$, the second term is bounded by $2c^{-1}\bar\gamma_1 n^{6/5} \varepsilon_n^\delta \leq \tilde{c}\varepsilon_n$ for some $\delta,\tilde{c}>0$ large enough. 

The remaining three terms are independent under the prior and so can be dealt with separately. By the assumptions on the base measure of the Dirichlet process, we have that $\sum_{i=0}^{\bar{N}} H(U_i)\leq  H(\mathbb{R}^+)$ and $H(U_i)= H(\mathbb{R}^+)\bar{H}(U_i)\gtrsim \lambda(U_i)/(b_n-a_n)^\eta \geq \varepsilon_n^{\delta}/(b_n-a_n)^\eta \geq \varepsilon_n^{\delta'}$ for $i=1,...,\bar{N}$ and some $\delta'>\delta$. For $i=0$, note that $\lambda (U_0) \geq (b_n-a_n)-\bar{N} \varepsilon_n^\delta \gtrsim 1$. Using the lower the bounds for the $\lambda(U_i)$, which come from the polynomial separation of the knots in Proposition \ref{prop:piecewise_approx}(iii), we can apply Lemma 10 of \cite{ghosal2007} to get
\begin{align}
\Pi_n\bigg(\bar{\gamma}_1 \sum_{i=0}^{\bar{N}} |p_i - \bar{p}_i|\leq \varepsilon_n\bigg)\gtrsim e^{-c\bar{N}\log (2\bar{\gamma}_1/\varepsilon_n)} \gtrsim e^{-c' n\varepsilon_n^2}.\label{eq: small_ball_p}
\end{align}
From the tail assumption on the density of $\gamma_1$ and the upper bound on $\bar{\gamma}_1$, we have
\begin{align}
\Pi_n\big(|\gamma_1-\bar{\gamma}_1|\leq \varepsilon_n\big)\gtrsim \varepsilon_n e^{-c(\bar{\gamma}_1+\varepsilon_n)^{1/(4+\mu)}}\geq  e^{-c' n^{1/5}}\geq e^{-n\varepsilon_n^2}.\label{eq: small_ball_g1}
\end{align}
By Proposition \ref{prop:piecewise_approx}(v), $|\bar\gamma_2|\lesssim n^{4/5}$, which, combined with the tail bound on the density of $\gamma_2$, yields
\begin{align}
\Pi_n((b_n-a_n)|\gamma_2-\bar{\gamma}_2|\leq \varepsilon_n) \gtrsim \tfrac{\varepsilon_n}{b_n-a_n} e^{-c(|\bar{\gamma}_2|+\varepsilon_n/(b_n-a_n))^{1/4}}\geq e^{-c' n^{1/5}}\geq e^{-n\varepsilon_n^2}\label{eq: small_ball_g2}
\end{align}
since $\varepsilon_n /(b_n-a_n) \rightarrow 0$ no faster than polynomially in $n$. Combining the above, we have that $\Pi_n (\|\bar{w}_n-W\|_{L^\infty([a_n,b_n])} \leq (3+\tilde{c})\varepsilon_n) \geq e^{-(2+c')n\varepsilon_n^2}.$
\end{proof}

\begin{proof}[Proof of Theorem \ref{thm: hier}]
Since the proof follows that of Theorem \ref{thm: contraction}, we only specify the details where the present proof differs. Using the same arguments, we restrict to studying posterior contraction based on observations arising from the log-concave density $f_{0,n} = f_0\psi_n/\int f_0\psi_n$ for $\psi_n(x) = \1_{[-t_n,t_n]}(x)$ with $t_n=a' \log n$ for some $a'>\alpha^{-1}$. We again apply Theorem \ref{thm:general contraction} so that we only need to show the small-ball probability \eqref{eq:small ball} for $\eps_n = (\log n)n^{-2/5}$.

Writing $\Pi_{a,b}$ for the prior conditional on $(a,b)$ and setting $\Delta_n = \{(a,b): -3t_n \leq a \leq -2t_n, 2t_n \leq b \leq 3t_n\}$, the small ball probability in \eqref{eq:small ball} is lower bounded by
\begin{equation}\label{eq:HB_lower_bd}
\begin{split}
& \int_{\Delta_n} \Pi_{a,b}\left( K(f_{0,n},f_W)\leq \varepsilon_n^2, \ V(f_{0,n},f_W) \leq \varepsilon_n^2 \right) \pi(a,b) da\, db \\
& \quad \geq \inf_{(a,b)\in\Delta_n} \Pi_{a,b}\left( K(f_{0,n},f_W)\leq \varepsilon_n^2, \ V(f_{0,n},f_W) \leq \varepsilon_n^2 \right) \times \int_{\Delta_n} \pi(a,b) da\, db.
\end{split}
\end{equation}
Using the lower bound assumption \eqref{eq:hier_cond} on $\pi(a,b)$, the last integral is lower bounded by $Ce^{-c_1(3t_n)^q - c_2(6t_n)^r} \int_{\Delta_n} da \, db \geq Ct_n^2 e^{-c_3 (\log n)^{q\vee r}} \gtrsim e^{-C'n\eps_n^2}$. It thus suffices to lower bound the infimum in the last display.

Since $[-\tfrac{8}{5\alpha} \log n,\tfrac{8}{5\alpha} \log n]\subset [a,b]$ and $(b - a) =O(\log n) = o(n^{4/5})$ for all $(a,b)\in\Delta_n$, we may apply Proposition \ref{prop:piecewise_approx} to construct an approximation $\bar{f}_n$ of $f_{0,n}$ based on the interval $[a,b]$ for any $(a,b)\in\Delta_n$. One can then proceed exactly as in the proof of Theorem \ref{thm: contraction} to lower bound the prior small-ball probability by $e^{-Cn\eps_n^2}$ for fixed $(a,b)\in \Delta_n$. Since all constants in that argument depend only on $\alpha$, $\beta$ and the prior hyperparameters, the lower bound is uniform over all $(a,b)$ with $(b-a) = O(n^{2/5})$ (to ensure $\eps_n$ in \eqref{def: eps} takes the value $(\log n)n^{-2/5}$) and $-a,b \geq \tfrac{8}{5\alpha} \log n$. In particular, the lower bound is uniform over $\Delta_n$.
\end{proof}

\begin{proof}[Proof of Corollary \ref{thm: approx}]
We use the notation employed in the proof of Theorem \ref{thm: contraction}. By \eqref{eq: KL_to_L_inf}, it suffices to lower bound the prior probability of an $L^\infty$-small ball about $\bar{w}_n$, where $\bar{f}_n = f_{\bar{w}_n}$ is the approximation \eqref{eq:approx_mixture_rep}. Since $\bar{N}\leq N$ (at least for $n$ large enough), we can add additional breakpoints to the piecewise linear function $\bar{w}_n$ with weights $\bar{p}_i=0$, $i=\bar{N}+1,...,N$, without changing $\bar{w}_n$. Without loss of generality, pick any such additional breakpoints to be no smaller than $cn^{-6/5}$.
Using similar computations to \eqref{eq: Linf_smallball}, for any $x\in[a_n,b_n]$,
\begin{align*}
|\bar{w}_n(x) - w(x)| & = \left| \bar{\gamma}_1 \sum_{i=1}^{N} \frac{z_i \wedge (x-a_n)}{z_i} \bar{p}_i - \bar{\gamma}_2 (x-a_n) -  \gamma_1\sum_{i=1}^N \frac{\theta _i\wedge (x-a_n)}{\theta_i}p_i - \gamma_2(x-a_n) \right| \\
& \leq |\bar{\gamma}_1-\gamma_1| \sum_{i=1}^N \frac{\theta_i \wedge (x-a_n)}{\theta_i}p_i + \bar{\gamma}_1 \left| \sum_{i=1}^N  \frac{\theta_i \wedge (x-a_n)}{\theta_i}p_i - \sum_{i=1}^N \frac{z_i \wedge (x-a_n)}{z_i} p_i  \right| \\
& \quad \quad + \bar{\gamma}_1 \sum_{i=1}^N \frac{z_i \wedge (x-a_n)}{z_i} |p_i -\bar{p}_i| + (b_n-a_n)|\bar{\gamma}_2-\gamma_2| \\
& \leq |\bar{\gamma}_1-\gamma_1| + \bar{\gamma}_1 \sum_{i=1}^N\frac{|\theta_i-z_i|}{\theta_i \wedge z_i } p_i  + \bar{\gamma}_1 \sum_{i=1}^N |p_i - \bar{p}_i| + (b_n-a_n)|\bar{\gamma}_2-\gamma_2|.
\end{align*}
The first and fourth terms are bounded from above by $\eps_n$ with prior probability at least $e^{-n\eps_n^2}$ by \eqref{eq: small_ball_g1} and \eqref{eq: small_ball_g2}, respectively, for both priors. For the second term note, similarly to the proof of Theorem \ref{thm: contraction}, that $z_1 \geq  cn^{-6/5}$. Taking $\theta_i\in[z_i,z_i+cn^{-6/5}\eps_n/\bar\gamma_1]$, the second term is bounded by $\bar\gamma_1\sum_{i=1}^N p_i \eps_n/\bar\gamma_1= \eps_n$. The probability of this set under the base measure is $H([z_i,z_i+cn^{-6/5}\eps_n/\bar\gamma_1])\gtrsim cn^{-6/5}\eps_n/\big(\bar\gamma_1(b_n-a_n)^{\eta}\big)$ by the assumptions on $\bar{H}$. The joint probability that $\theta_i\in [z_i,z_i+cn^{-6/5}\eps_n/\bar\gamma_1]$ for every $i=1,...,N$ is therefore bounded from below by a multiple of $(cn^{-6/5}\eps_n/\bar\gamma_1)^N/(b_n-a_n)^{\eta N}\gtrsim e^{-c_1 N\log n}\geq e^{-c_2 n\eps_n^2}$, for sufficiently large constants $c_1,c_2>0$.

It remains to show that the third term is bounded from above by $\eps_n$ with probability at least $e^{-c n\eps_n^2}$ for some $c>0$.  In the case where $(p_1,...,p_N)$ is endowed with a Dirichlet distribution, this statement follows from \eqref{eq: small_ball_p}. In the case of the truncated stick-breaking prior, writing $(\bar{p}_1,\dots,\bar{p}_N)$ in decreasing order, we note that there exist $0\leq\bar{v}_1,...,\bar{v}_{N-1}\leq 1$ such that $\bar{p}_i=\prod_{j=1}^{i-1}(1-\bar{v}_{j})\bar{v_i}$, $i=1,...,N-1$, and $\bar{p}_N=\prod_{j=1}^{N-1}(1-\bar{v}_{j})$. Define for $i=1,2,...,N-1$ the intervals
$$I_i=\big[(\bar{v}_i-\eps_n n^{-4/5}N^{-2})\vee \eps_n n^{-4/5}N^{-2}/2,(\bar{v}_i+\eps_n n^{-4/5}N^{-2})\wedge (1-\eps_n n^{-4/5}N^{-2}/2)\big].$$
For $v_i\in I_i\subset[0,1]$, $i=1,..,N-1$, we have
\begin{align*}
\Big| (1-v_1)...(1-v_i)v_{i+1}-(1-\bar{v}_1)...(1-\bar{v}_i)\bar{v}_{i+1}\Big|&\leq
|v_1-\bar{v}_1|+|v_2-\bar{v}_2|+...+|v_{i+1}-\bar{v}_{i+1}|\\
&\leq (i+1)\eps_n n^{-4/5}N^{-2}\leq \eps_n n^{-4/5}N^{-1}.
 \end{align*}
Hence for $p_i:=(1-v_1)(1-v_2)...(1-v_{i-1})v_i$, $i=1,...,N-1$, $(p_1,..,p_N)$ is in the $N$-dimensional simplex and $\bar\gamma_1\sum_{i=1}^N |\bar{p}_i-p_i|\leq \bar\gamma_1 N\eps_n n^{-4/5}N^{-1}\lesssim\eps_n$. Finally, we note that for $v_i\sim \text{Beta}(a,b)$, we have $P(v_i\in I_i)\gtrsim(\eps_n n^{-4/5}N^{-2})^{a\vee b}$ and we can therefore conclude
\begin{align*}
P\bigg(\bar\gamma_1\sum_{i=1}^N|p_i-\bar{p_i}|\leq c\eps_n\bigg)&\geq \prod_{i=1}^{N-1} P( v_i\in I_i) \geq e^{N(a\vee b)\log(\eps_n N^{-2}n^{-4/5})}\geq e^{-c_1N\log n}\geq e^{-c_2n\eps_n^2},
\end{align*}
for some large enough constants $c_1,c_2>0$, thereby completing the proof.

For the hierarchical case where we assign a prior to $(a,b)$, the proof follows as in that of Theorem \ref{thm: hier} using \eqref{eq:HB_lower_bd} and the lower bound for the small-ball probability just derived.
\end{proof}
%
%
%

\section{Technical results}\label{sec:technical proofs}

\subsection{Proof of Proposition \ref{prop:piecewise_approx}}\label{sec:linear proofs}

In this section, we construct the piecewise log-linear approximation for an upper semi-continuous log-concave density given in Proposition \ref{prop:piecewise_approx}. In particular, we require that the number of knots in the approximating function does not grow too quickly and that the knots are polynomially separated, thereby rendering the construction somewhat involved. The proof relies on firstly approximating any continuous concave function on a given compact interval using a piecewise linear function. One then splits $\supp(f_0)$ into sets, depending on the size of both $\log f_0$ and $|(\log f_0)'|$, and obtains suitable piecewise linear approximations defined locally on each of these sets. Piecing together these local functions gives the desired global approximation.

We now construct a piecewise linear approximation of a continuous concave function $w$ on a compact interval $[a,b]$. For any partition $a=x_0<x_1<\dots<x_m=b$ of $[a,b]$, let $\widetilde w_m$ denote the piecewise linear approximation of $w$ given by
\begin{equation}\label{eq:fmbar}
  \widetilde w_m(x):=\sum_{i=2}^{m} \bigg(\frac{x-x_{i-1}^*}{x_i^*-x_{i-1}^*}\frac{1}{x_i-x_{i-1}}\theta_i+\frac{x_i^*-x}{x_i^*-x_{i-1}^*}\frac{1}{x_{i-1}-x_{i-2}}
  \theta_{i-1}\bigg)\1_{(x_{i-1}^*, x_i^*]}(x),
   \end{equation}
   where $\theta_i:=\int_{x_{i-1}}^{x_i}w(s)ds$ and $x_i^*:=\frac{x_i+x_{i-1}}{2}$. On $[a, x_1^*]$ and $(x_m^*, b]$, the function is defined by linearly extending the piecewise linear function defined above, that is
   \begin{equation}\label{eq:fmbara}
   \begin{split}
   & \widetilde w_m(a) := \frac{1}{x_2^*-x_1^*} \Big(\frac{x_2^*-a}{x_1-a}\theta_1 - \frac{x_1^*-a}{x_2-x_1}\theta_2 \Big), \\
    & \widetilde w_m(b) := \frac{1}{x_m^*-x_{m-1}^*} \Big(  \frac{b-x_{m-1}^*}{b-x_{m-1}}\theta_m - \frac{b-x_m^*}{x_{m-1}-x_{m-2}}\theta_{m-1} \Big).
    \end{split}
   \end{equation}
The function $\widetilde w_m$ takes value $\widetilde w_m(x_i^*) = \frac{1}{x_i-x_{i-1}} \int_{x_{i-1}}^{x_i} w(s) ds$ at the midpoint $x_i^*=\frac{x_i+x_{i-1}}{2}$ of the interval $[x_{i-1},x_i]$ and interpolates linearly in between.

\begin{lemma}\label{lemma:concave}
Let $w:[a,b]\to\R$ be a continuous concave function, where $-\infty<a<b<\infty$. For any partition $a=x_0<x_1<\dots<x_m=b$ of $[a,b]$, let $\widetilde w_m$ denote the piecewise linear approximation of $w$ defined in \eqref{eq:fmbar} and \eqref{eq:fmbara}. Then $\widetilde w_m$ is a concave function.
\end{lemma}

\begin{proof}
By rescaling, we may without loss of generality assume that $[a,b]=[0,1]$. Note that $\widetilde w_m$ is concave if and only if
\begin{equation}\label{eq:conc}
\widetilde w_m(x_i^*)\geq \frac{x_i^*-x_{i-1}^*}{x_{i+1}^*-x_{i-1}^*} \widetilde w_m(x_{i+1}^*)+\frac{x_{i+1}^*-x_i^*}{x_{i+1}^*-x_{i-1}^*}\widetilde w_m(x_{i-1}^*)
\end{equation}
for $i=2,\dots,m-1$. Indeed, since $\widetilde w_m$ is piecewise linear, it is concave if and only if at every point where the derivative is discontinuous (i.e. a knot), the left derivative is greater than or equal to the right derivative. The above statement follows since the derivative of $\widetilde w_m$ is discontinuous (at most) at the points $x_i^*$, $i=2,\dots,m-1$, where the desired inequality is:
$$
\frac{\widetilde w_m(x_{i+1}^*) - \widetilde w_m(x_i^*)}{x_{i+1}^*-x_i^*} \leq \frac{\widetilde w_m(x_{i}^*) - \widetilde w_m(x_{i-1}^*)}{x_{i}^*-x_{i-1}^*},
$$
which is equivalent to \eqref{eq:conc}.

To see that \eqref{eq:conc} holds if $w$ is concave, we argue by contradiction and suppose that there exists $i$ such that $\widetilde w_m(x_i^*)< \frac{x_i^*-x_{i-1}^*}{x_{i+1}^*-x_{i-1}^*}\widetilde w_m(x_{i+1}^*)+\frac{x_{i+1}^*-x_i^*}{x_{i+1}^*-x_{i-1}^*}\widetilde w_m(x_{i-1}^*)$.
Consider the linear function $l$,
$$l(x):=\frac{x-x_{i-1}^*}{x_{i+1}^*-x_{i-1}^*}\widetilde w_m(x_{i+1}^*)+\frac{x_{i+1}^*-x}{x_{i+1}^*-x_{i-1}^*}\widetilde w_m(x_{i-1}^*).$$
In particular, we have that $\widetilde w_m(x_{i-1}^*)-l(x_{i-1}^*)=\widetilde w_m(x_{i+1}^*)-l(x_{i+1}^*)=0$ and $\widetilde w_m(x_i^*)<l(x_i^*)$. We further denote $g:=w-l$ and observe that 
\begin{align*}
\tilde g_m(x_i^*)&=\frac{1}{x_i-x_{i-1}}\int_{x_{i-1}}^{x_i} \big(w(s)-l(s)\big)ds=\widetilde w_m(x_i^*)-\int_{x_{i-1}}^{x_i} l(s)ds=\widetilde w_m(x_i^*)-l(x_i^*).
\end{align*}
It follows that $\tilde g_m(x)=\widetilde w_m(x)-l(x)$ for all $x\in[0,1]$ and hence by the mean value theorem, 
$$\tilde g_m(x_i^*)=\frac{1}{x_i-x_{i-1}}\int_{x_{i-1}}^{x_i} g(s)ds=g(\xi_i)$$
for some $\xi_i\in [x_{i-1},x_i]$.
One can similarly prove the existence of two points, $\xi_{i-1}\in [x_{i-2},x_{i-1}]$ and $\xi_{i+1}\in [x_i,x_{i+1}]$, such that $\tilde g_m(x_{i-1}^*)=g(\xi_{i-1})$ and $\tilde g_m(x_{i+1}^*)=g(\xi_{i+1})$. Using the above results, we deduce the existence of three points $\xi_{i-1}<\xi_i<\xi_{i+1}$ such that $g(\xi_{i-1})=0=g(\xi_{i+1})$ and $g(\xi_i)<0$, which is a contradiction since $g$ is concave by the concavity of $w$ and $l$.
\end{proof}

\begin{lemma}\label{lemma:linear}
Let $w: [a,b]\rightarrow\mathbb{R}$ be a continuous concave function with $w_+'(a)-w_-'(b)\leq M$ and where $-\infty<a<b<\infty$. Then there exists a partition $a=x_0<x_1<\dots<x_m=b$ of $[a,b]$ with 
$\min_{i=1,\dots,m}(x_i-x_{i-1})\geq (b-a)(2m)^{-2}$ and such that
\begin{equation*}\label{eq:infnorm}
\sup_{x\in[a,b]}|w(x)-\widetilde w_m(x)|\leq C\frac{M(b-a)}{m^2},
\end{equation*}
where $\widetilde w_m$ is the piecewise linear approximation of $w$ defined in \eqref{eq:fmbar} and \eqref{eq:fmbara} and $C>0$ is a universal constant (i.e. not depending on $a,b,m$).
\end{lemma}

\begin{proof}
By translation we may without loss of generality take $a=0$. Recall that since $w$ is a continuous concave function, it has left and right derivatives at every point $x\in[0,b]$. Define $\Delta w'(x)=w_-'(x)-w_+'(x).$ For every $r\geq 1$, let $\mathcal P_1:=\{x_{i,1}:=\frac{ib}{r}, i=0,\dots,r\}$ be the uniform partition of $[0,b]$ and let $\tilde x_{i,2}$, $i=1,\dots,r_2$, be the points such that $\Delta w'(\tilde x_{i,2})\geq M/r$, setting $r_2=0$ if no such point exists. By concavity of $w$,
\begin{align}\label{eq:sim_arg}
M\geq w_+'(0)-w_-'(b)&\geq  \sum_{i=1}^{r_2} \Delta w'(\tilde x_{i,2})\geq \frac{Mr_2}{r},
\end{align}
so that $r_2\leq r$.

Consider a new partition $\mathcal P_2:=\{x_{0,2}<\dots<x_{r_2',2}\}$ of $[0,b]$, consisting of the points $\{x_{i,1}\} \cup \{\tilde x_{i,2}\} \cup \{\tilde x_{i,2} - br^{-2}\} \cup \{\tilde x_{i,2} +br^{-2}\}$ written in increasing order. Note that $r_2' \leq r+3r_2$. Colour in red all the points of the form $\tilde x_{i,2}$ and $\tilde x_{i,2}-br^{-2}$, so that each red point is the left endpoint of an interval of length at most $br^{-2}$. This colouring will be used to keep track of points that have a close neighbour on the right.

We next refine the partition $\mathcal P_2$ by adding the point $y$ between $x_{i,2}$ and $x_{i+1,2}$
\begin{equation*}\label{eq:xi3}
y:=\sup \Big\{x>x_{i,2}: w_+'(x_{i,2})-w'_-(x)\leq \frac{2M}{r}\Big\}
\end{equation*}
if
\begin{equation*}
w_+'(x_{i,2})-w_-'(x_{i+1,2})>\frac{2M}{r}.
\end{equation*}
Denote by $r_3$ the total number of points $y$ added in this manner to the sequence. We further add the points $y-br^{-2}$, $y+br^{-2}$ and colour in red all points of the form $y$ and $y-br^{-2}$, similarly to the previous case. Repeating this procedure results in a new partition that separates intervals where the derivative decreases by at most $2M/r$. Denote by $\mathcal P_3:=\{0=x_{0,3}<x_{1,3}<\dots<x_{r_3',3}=b\}$ this new partition. We now show that $r_3'\leq 7 r$.

Let $y$ by any point added in the way just described. Suppose by contradiction that $w_+'(x_{i,2}) - w_-'(y) < M/r$. By definition, we know that for all $x > y$, $w_+'(x_{i,2}) - w_-'(x) > 2M/r$. Subtracting the two inequalities gives $w_-'(y) - w_-'(x) > M/r$. However, since the right derivative of a concave function is right continuous, taking the limit $x \to y^+$ (and restricting to the points $x$ where $w$ is differentiable) yields $\Delta w'(y) \geq M/r$.  This is a contradiction however, because if this were the case, $y$ would already belong to $\mathcal P_2$. Since $w_+'(x_{i,2}) - w_-'(y) \geq M/r$, using a similar argument to \eqref{eq:sim_arg} gives $r_3 \leq r$ so that $r_3'\leq 7 r$.

Finally, if the function $w$ is not differentiable at the point $x_{i,3}^*=\frac{x_{i,3}+x_{i-1,3}}{2}$, we split $[x_{i-1,3},x_{i,3}]$ into two parts in such a way that $w$ is differentiable at the midpoints of both new intervals and each interval has size at least $(x_{i,3}-x_{i-1,3})/3$. We add the points separating the new intervals to the previous partition, thereby obtaining $\mathcal P_4:=\{0=x_{0,4}<x_{1,4}<\dots<x_{\nu,4}=b\}$. The cardinality of $\mathcal P_4$, satisfies $\nu+1 \leq 14r+1$. We now create a new partition $\mathcal P$ with polynomially separated points using the following algorithm.
\begin{enumerate}
\item Set $\mathcal{P} = \mathcal{P}_4$, keeping track of all the points coloured red. Set $\tilde{x} = x_{0,4}$.
\item If $b-\tilde{x} \leq br^{-2}$, remove all points in $\mathcal{P}$ strictly between $\tilde{x}$ and $b$ skip to Step 4.
\item Set $y = \inf\{ t\in \mathcal{P}: t > \tilde{x} + br^{-2} \}$. Remove all elements of $\mathcal{P}$ between $\tilde{x}$ and $y$. If at least one element was removed, add to $\mathcal{P}$ the point $s=\tilde{x}+br^{-2}+\varepsilon$ for some $0<\varepsilon < br^{-2}\wedge (y-\tilde{x}-br^{-2})$ such that $w$ is differentiable at $(s+y)/2$ and $(s+\tilde{x})/2$. Colour $\tilde{x}$ red to mark that $s-\tilde{x}<2br^{-2}$. Set $\tilde{x} := s$. If no point was removed from $\mathcal{P}$, set $\tilde{x}:= y$. Go to Step 2.
\item If $\tilde{x}=b$ then stop. Otherwise set $y = \max \{ t\in \mathcal{P}: t< \tilde{x} \}$ and remove $\tilde{x}$ from $\mathcal{P}$. If $b-y > 2br^{-2}$, add the point $s:=b-br^{-2}-\varepsilon$ to $\mathcal{P}$ and colour it red, where $0 < \varepsilon < (b-y-2br^{-2}) \wedge br^{-2}$ is such that $w$ is differentiable at $(y+s)/2$. If $b-y \leq 2br^{-2}$, add the point $s := (y+b)/2$ and colour both $y$ and $s$ red.
\end{enumerate}

Relabel the final partition $\mathcal P:=\{0=x_0<x_1<\dots<x_m=b\}$ and note that
\begin{equation*}
r \leq m \leq \nu \leq 14r+1.
\end{equation*}
By construction $\min_{i=0,\dots,m-1}(x_{i+1}-x_{i})\geq \tfrac{1}{2}br^{-2} \geq \tfrac{1}{2}bm^{-2}$ and 
\begin{equation*}
x_{i+1}-x_i\leq \begin{cases}
2C_0^2 bm^{-2} &\text{ if } x_i \text{ is coloured red (with $C_0=15$}),\\
C_0 bm^{-1}  &\text{ otherwise,}
\end{cases}
\end{equation*}
since if $x_i$ is coloured red, $x_{i+1}-x_i \leq 2br^{-2}\leq 2C_0^2 bm^{-2}$.

We now show that $\|w-\widetilde w_m\|_{\infty}=O(b m^{-2})$. If $x_{i-1}$ is red, then by the mean value theorem, there exists $\xi_i\in J_i:=[x_{i-1},x_i]$ such that $\widetilde w_m(x_i^*)=w(\xi_i)$. Using the Lipschitz continuity of $w$ and $\widetilde w_m$,
\begin{equation*}\label{eq:smallintervals}
|w(x)-\widetilde w_m(x)|\leq 2 C_0^2\frac{bM}{m^2}, \quad \forall x\in J_i:=[x_{i-1},x_i].
\end{equation*}

If $x_{i-1}$ is not red, Taylor expanding $w$ at the points $x_i^*:=\frac{x_i+x_{i-1}}{2}$ (at which $w$ is differentiable by the construction of $\mathcal P)$ gives
\begin{align}\label{eq:taylor}
 w(x)&=w(x_i^*)+w'(x_i^*)(x-x_i^*)+R_i(x), \quad x\in J_i.
\end{align}
Due to the construction of $\mathcal P$,
 \begin{align*}
  |R_i(x)|&=\Big|w(x)-w(x_i^*)-w'(x_i^*)(x-x_i^*)\Big|\nonumber \\
    &=\big|(w'(\xi_i)-w'(x_i^*)) (x-x_i^*)\big|, \label{eq:resto}
 \end{align*}
where $w'(\xi_i)$ stands here for some value in the interval $[w'_+(\xi_i), w'_-(\xi_i)]$ for some point $\xi_i\in J_i$. We then deduce that $|R_i(x)|\leq \frac{2M}{r} \frac{C_0b}{2m}\leq C_0^2bM/m^2.$

Since $\widetilde w_m$ is piecewise linear, we can write
$$\widetilde w_m(x)=\widetilde w_m(x_i^*)+\widetilde w_m'(x_i^*)(x-x_i^*), \quad \quad x\in J_i,$$
where $\widetilde w_m'$ denotes the left or right derivative of $\widetilde w_m$ at $x_i^*$, depending on whether $x<x_i^*$ or $x>x_i^*$.
We now show that $|\widetilde w_m'(x_i^*)-w'(x_i^*)|\leq \max\{w_+'(x_{i-1})-w'(x_i^*),w'(x_i^*)-w_-'(x_{i+1})\}$ for $i=1,\dots,m-1$. Consider the case of right derivatives (the same argument also works for left derivatives). Using the definition of $\widetilde w_m$ and that $\theta_i=\int_{J_i}w(s)ds$,
\begin{align*}
\widetilde w_{m,+}'(x_i^*)&=\frac{\widetilde w_m(x_{i+1}^*)-\widetilde w_m(x_i^*)}{x_{i+1}^*-x_i^*}=\frac{1}{x_{i+1}^*-x_i^*}\bigg(\frac{\theta_{i+1}}{x_{i+1}-x_i}-\frac{\theta_{i}}{x_{i}-x_{i-1}}\bigg)\\
&=\frac{1}{(x_{i}-x_{i-1})(x_{i+1}^*-x_i^*)}\int_{x_{i-1}}^{x_i}\bigg[w\Big(\frac{x_{i+1}-x_i}{x_i-x_{i-1}}t+\frac{x_{i}^2-x_{i+1}x_{i-1}}{x_i-x_{i-1}}\Big)-w(t)\bigg]dt
\end{align*}
and
$$\int_{x_{i-1}}^{x_i}\bigg[\frac{x_{i+1}-x_i}{x_i-x_{i-1}}t+\frac{x_{i}^2-x_{i+1}x_{i-1}}{x_i-x_{i-1}}-t\bigg]dt=(x_{i+1}^*-x_i^*)(x_{i}-x_{i-1}).$$
By the continuity and concavity of $w$, $(v-u) w_-'(x_{i+1}) \leq w(v)-w(u) \leq (v-u) w_+'(x_{i-1})$ for any $x_{i-1} \leq u\leq v\leq x_{i+1}$. Combining all of the above yields
\begin{equation}
w_-'(x_{i+1})\leq \widetilde w_{m,+}'(x_i^*)\leq w_+'(x_{i-1}).
\label{eq:lin_approx_derive_control}
\end{equation}
We remark that $\max\{w_+'(x_{i-1})-w'(x_i^*),w'(x_i^*)-w_-'(x_{i+1})\}\leq \frac{5M}{r}$. Indeed, since $x_i-x_{i-1}> C_0^2b/m^2$, the point $x_i$ is not equal to $x_{j,4}'$ for any $j$ and hence both $\Delta w'(x_i)<M/r$ and $w_+'(x_i)-w_-'(x_{i+1})\leq 2M/r$ hold. Together with $w_+'(x_{i-1})-w'(x_i^*)\leq 2M/r$ and $w'(x_i^*)-w'_-(x_i)\leq 2M/r$, this verifies the preceding statement. Then
\begin{align*}
w'(x_i^*)-w_-'(x_{i+1})&\leq w'(x_i^*)-w'_-(x_i)+w_-'(x_i)-w_+'(x_i)+w_+'(x_i)-w_-'(x_{i+1})\leq \frac{5M}{r}.
\end{align*}
We hence deduce that
$
 |w'(x_i^*)-\widetilde w_m'(x_i^*)|\leq 5 M C_0 m^{-1}.
 $
Finally, using \eqref{eq:taylor} and the fact that $\int_{J_i}(x-x_i^*)dx=0$,
$$\big|w(x_i^*)-\widetilde w_m(x_i^*)\big|=\frac{1}{x_i-x_{i-1}}\bigg|\int_{J_i}\big(w(x_i^*)-w(x)\big)dx\bigg|\leq \sup_{x\in J_i}|R_i(x)|\leq \frac{C_0^2 bM}{m^2}.$$
Collecting together all the pieces, we have that for any $x\in[x_{i-1},x_i]$,
\begin{align*}
 |w(x)-\widetilde w_m(x)|\leq  |w(x_i^*)-\widetilde{w}_m(x_i^*)|+|x-x_i^*| |\widetilde w_m'(x_i^*)-w'(x_i^*)| +|R_i(x)| \leq\frac{9}{2} C_0^2Mb m^{-2}.
\end{align*}
\end{proof}

\begin{lemma}\label{lem: mix_rep_lem}
Any piecewise linear concave function $w:[a,b] \rightarrow \R$ with $N$ knots $\{z_1,...,z_N\}$ can be written in the form
\begin{align*}
w(x)= \gamma_1\sum_{i=1}^{N}\frac{z_i \wedge (x-a)}{z_i} p_i-\gamma_2 (x-a)+\gamma_3,
\end{align*}
with parameters $0\leq \gamma_1\leq (w_+'(a)-w_{-}'(b))(b-a)$, $|\gamma_2|\leq  |w_{-}'(b)|$,  $\gamma_3\in\mathbb{R}$, $\sum_{i=1}^{N} p_i=1$ and $p_i\geq 0$ for  $i=1,\dots,N$.
\end{lemma}

\begin{proof}
The left derivative of $w$ is a step function $g:\, (a,b]\mapsto \mathbb{R}$ with $g(a+\eps)=w_+'(a)$ for sufficiently small $\eps>0$ and $g(b)=w_-'(b)$. By shifting this function vertically by  $-w_-'(b)$, we arrive at a non-negative, bounded, monotone decreasing step function, which can therefore be written as a monotone decreasing probability density times a normalizing constant $\gamma_1$. It is easy to see that $\gamma_1\leq (b-a) (w_+'(a)-w_-'(b))$. The step function can therefore be represented as
$$g(x)=\gamma_1\int_{x-a}^{b-a}\frac{1}{z}dP_{N}(z)+w_-'(b),\qquad x\in[a,b],$$
with $P_N$ an atomic probability measure with $N$ atoms on $[0,b-a]$ and $\gamma_1$ the normalizing constant. Integrating the step function $g$ yields
\begin{align*}
\bar{w}_N(x)
=\gamma_1\int_0^{b-a}\frac{z\wedge (x-a)}{z }dP_{N}(z)+ w_-'(b) (x-a) +C, 
\end{align*}
which is equal to $w$ for an appropriately chosen constant $C>0$.
\end{proof}

\begin{proof}[Proof of Proposition \ref{prop:piecewise_approx}]
Let $\psi_n(x) = \1_{[-s_n,s_n]}(x)$ for $s_n = \tfrac{4}{5\alpha}\log n$. The log-concave density function $f_1 = f_{1,n}=f_0 \psi_n /\int f_0 \psi_n$ supported on $[-s_n,s_n]$ satisfies $|1-\int f_0 \psi_n| \leq P_{f_0}[-s_n,s_n]^c \leq 2e^\beta \alpha^{-1} n^{-4/5}$. Arguing as in \eqref{eq:log concave compact support approx}, one has $h^2(f_0,f_{1,n}) \leq 12e^\beta \alpha^{-1} n^{-4/5}$ for $n\geq (4e^\beta/\alpha)^{5/4}$.

We write $f_1 = e^{w_1}$ and construct the approximating function $\bar{f}_n$ according to the value of $w_1$ and its left and right derivatives $w_{1,-}'$ and $w_{1,+}'$. Let
\begin{align*}
& A_0^n = \{ x\in[a_n,b_n]:w_1(x)< -\tfrac 45 \log n \},\\
& A_1^n = \{x\in[a_n,b_n]: w_1(x) \geq -\tfrac 45 \log n, |w_{1,\pm}'(x)| > n^{4/5} \},\\
& A_{2,j}^n = \{x\in[a_n,b_n]: w_1(x) \geq -\tfrac45 \log n, 2^{-j-1} n^{4/5} < |w_{1,\pm}'(x)|\leq 2^{-j} n^{4/5} \}, \quad j = 0,\dots, j_n,  \\
& A_3^n = \{x\in[a_n,b_n]: w_1(x) \geq -\tfrac 45 \log n, |w_{1,\pm}'(x)| \leq D \},
\end{align*}
where $D>0$ is some fixed constant, $|w_{1,\pm}'(x)| = \max (|w_{1,+}'(x)|,|w_{1,-}'(x)|)$ and $j_n = \lceil \log_2 (n^{4/5}/D)\rceil -1$. In fact the set where the left and right derivatives of the concave function $w_1$ do not agree has measure zero. Note that the above sets are all disjoint except $A_{2,j_n}^n$ and $A_3^n$: since $j_n$ is the smallest integer such that $2^{-j_n-1}n^{4/5} \leq D$, these last two sets may overlap. In particular, we can express $[a_n,b_n]$ as the almost disjoint union of the above sets. Write $B_n = (\cup_{j=0}^{j_n} A_{2,j}^n) \cup A_3^n \subset [a_n,b_n]$ and note that by the concavity of $w_1$, this is an interval. Since $\|f_1\|_\infty \leq 2e^{\beta}$ for $n\geq (4e^\beta/\alpha)^{5/4}$, the set $A_1^n$ consists of at most two intervals, each of width $O(n^{-4/5}\log n)$. Using again the boundedness of $f_1$, the definition of $A_0^n$ and that $|\text{supp}(f_1)|\lesssim \log n$,
\begin{equation}\label{eq:int_Bm}
\int_{B_n} f_1dx = 1 - O(n^{-4/5}\log n),
\end{equation}
so that in particular, $B_n \neq \emptyset$ for $n\geq n_0(\alpha,\beta)$ large enough.

We now construct a partition $\Pp_n$ of $B_n$ based on which we take the piecewise linear approximation \eqref{eq:fmbar}-\eqref{eq:fmbara} of the function $w_1$. Note that $A_{2,j}^n$ consists of at most two disjoint intervals, $A_{2,j,+}^n$ and $A_{2,j,-}^n$, which by the boundedness of $f_1$ are each of length $O(2^{j+1}n^{-4/5}\log n)$. Let $\Pp_{n,A_{2,j,+}^n}$ denote the partition of the interval $A_{2,j,+}^n$ given by Lemma \ref{lemma:linear} with partition size $m = \lceil 2^{-j/2} n^{3/5} |A_{2,j,+}^n|^{1/2}/\sqrt{\log n} \rceil =  O(n^{1/5})$ and let $\Pp_{n,A_{2,j,-}^n}$ be the analogous partition constructed on $A_{2,j,-}^n$. Similarly, $A_3^n \cap (A_{2,j_n}^n)^c$ consists of a single interval of length $O(\log n)$. Let $\Pp_{n,A_3^n \cap (A_{2,j_n}^n)^c}$ denote the corresponding partition of $A_3^n \cap (A_{2,j_n}^n)^n$ given by Lemma \ref{lemma:linear} with partition size $m = \lceil n^{1/5} (D|A_3^n \cap (A_{2,j_n}^n)^c|/\log n)^{1/2} \rceil = O(n^{1/5}).$ Define the overall partition
$$ \mathcal{P}_n = \mathcal{P}_{n,A_3^n \cap (A_{2,j_n}^n)^c} \cup \bigcup_{j=0}^{j_n} (\mathcal{P}_{n,A_{2,j,+}^n} \cup \mathcal{P}_{n,A_{2,j,-}^n})$$
of $B_n$, which has $O(j_n n^{1/5}) = O(n^{1/5} \log n)$ points. The associated piecewise linear function $\widetilde{w}_n$ defined in \eqref{eq:fmbar}-\eqref{eq:fmbara} based on $\mathcal{P}_n$ is concave by Lemma \ref{lemma:concave} and by construction corresponds to the partition given in Lemma \ref{lemma:linear} for each of the sets comprising $B_n$. It therefore satisfies the conclusions of Lemma \ref{lemma:linear} on each such set (with the appropriate $m$), so that in particular,
\begin{itemize}
\item $\|w_1 - \widetilde{w}_n\|_{L^\infty(A_{2,j}^n)}\leq Cn^{-2/5}\log n$ for some universal constant $C>0$ independent of $j$,
\item the partition points in $A_{2,j}^n$ are distance at least $c2^j n^{-6/5}\log n \geq c n^{-6/5}\log n$ apart for some universal constant $c>0$ independent of $j$,
\item on $A_3^n \cap (A_{2,j_n}^n)^c$, we have the same $L^\infty$-bound with the partition points being $cn^{-2/5}\log n$-separated.
\end{itemize}
Moreover, since these intervals meet only at their boundaries, and the boundary points of the intervals are contained in the partition presented in Lemma \ref{lemma:linear}, the interval boundaries will be contained in $\mathcal{P}_n$. Consequently, the separation property continues to hold even across the different subpartitions. In conclusion, we have shown that $\widetilde{w}_n$ is concave and piecewise linear with $O(n^{1/5}\log n)$ knots, which are $cn^{-6/5}\log n$-separated, and satisfies 
\begin{align}
\sup_{x\in B_n} |\widetilde{w}_n(x) - w_1(x)| = O(n^{-2/5}\log n).
\label{eq:l_infty_bd_intermediate}
\end{align}

We now extend the approximating function to $[a_n,b_n] \supset B_n$. Write $\mathcal{P}_n = (x_i)_{i=0}^M$, where $\min (B_n) = x_0<x_1<...<x_M = \max(B_n)$ and $M = O(n^{1/5}\log n)$. Define $\bar{w}_n:[a_n,b_n]\rightarrow \R$ as
\begin{align}\label{eq: concave}
\bar{w}_n(x)=
\begin{cases} 
      \widetilde{w}_n(x_0) + \big(w_{1,-}'(x_0)\wedge n^{4/5}\vee  (-n^{4/5})\big) (x-x_0) & x \in [a_n,x_0],\\
      \widetilde{w}_n(x) & x\in B_n, \\
       \widetilde{w}_n(x_M) + \big(w_{1,+}'(x_M)\wedge n^{4/5}\vee  (-n^{4/5}) \big)(x-x_M) & x \in [x_M,b_n].
   \end{cases}
\end{align}
This is simply the function $\widetilde{w}_n$ extended linearly from the boundary points of $B_n$ with slope $w_{1,-}'(x_0)\wedge n^{4/5}\vee (-n^{4/5}) $ and $w_{1,+}'(x_M)\wedge n^{4/5}\vee  (-n^{4/5})$ on $[a_n,x_0]$ and $[x_M,b_n]$ respectively. We now verify that $\bar{w}_n$ is concave, for which it is enough to show that $\bar{w}_{n,+}'(x_0) \leq \bar{w}_{n,-}'(x_0)$ and $\bar{w}_{n,+}'(x_M) \leq \bar{w}_{n,-}'(x_M)$. For the first inequality, using \eqref{eq:lin_approx_derive_control}, the concavity of $w_1$ and the boundary construction of $\widetilde{w}_n$ given by \eqref{eq:fmbar}, $\bar{w}_{n,+}'(x_0) = \widetilde{w}_{n,+}'(x_0) = \widetilde{w}_{n,+}'(x_1^*)\leq w_{1,+}'(x_0) \leq w_{1,-}'(x_0)$. Since $x_0\in B_n$, it also holds that $|\bar{w}_{n,+}'(x_0)|\leq n^{4/5}$. The second inequality can be proved analogously.

Since $\log (1+z) = O(z)$ as $|z| \rightarrow 0$, it follows that $\log \int f_0 \psi_n = O(n^{-4/5})$. Using this and \eqref{eq:l_infty_bd_intermediate},
\begin{equation*}
|\bar{w}_n(x_0) - \log f_0 (x_0)| \leq \Big| \log \int f_0 \psi_n \Big| + O(n^{-2/5}\log n)= O(n^{-2/5}\log n).
\end{equation*}
By concavity, the slope of the linear extension on $[a_n,x_0]$ satisfies $w_{1,-}'(x_0) = (\log f_0)_{-}'(x_0) \leq (\log f_0)_{+}'(x)$ for all $x< x_0$ such that $f_0(x) >0$. Combining the above yields $\bar{w}_n(x) \geq \log f_0(x) - O(n^{-2/5}\log n)$ for all $x\in [a_n,x_0]$. The same computation also gives the result for $x\in [x_M,b_n]$, so that for some $C>0$,
\begin{align}
\sup_{x\in [a_n,b_n]\backslash B_n} (\log f_0(x) - \bar{w}_n(x)) \leq C n^{-2/5} \log n.
\label{eq:tail_control}
\end{align}

Define the log-concave density
\begin{align*}
\bar{f}_n(x) =
\begin{cases} 
      e^{\bar{w}_n(x)} / \int_ {a_n}^{b_n} e^{\bar{w}_n} & x\in [a_n,b_n], \\
      0 & x \not\in [a_n,b_n].
   \end{cases}
\end{align*}
This function is piecewise log-linear, has $O(n^{1/5}\log n)$ knots and satisfies $(ii)$ and $(iii)$ by construction. We have
\begin{equation}\label{eq:hell_decom}
h^2(f_1,\bar{f}_n) \leq 2\int_{B_n^c} f_1  + 2\int_{B_n^c} \bar{f}_n  +  \int_{B_n} (f_1^{1/2} - \bar{f}_n^{1/2})^2 .
\end{equation}

The first integral is $O(n^{-4/5}\log n)$ by \eqref{eq:int_Bm}. Using \eqref{eq:l_infty_bd_intermediate}, $\int_{B_n} e^{\bar{w}_n} = e^{o(1)} \int_{B_n} f_1$.
Write the second integral as $\int_{B_n^c} \bar{f}_n = \int_{A_0^n} \bar{f}_n + \int_{A_1^n} \bar{f}_n$. By the definition of $A_0^n$ and \eqref{eq:tail_control},  $\int_{A_0^n} e^{\bar{w}_n} \leq \int_{A_0^n} e^{-(4/5)\log n + n^{-2/5} \log n} \leq  (x_0-a_n + b_n-x_M) e^{o(1)}n^{-4/5} = O((b_n-a_n) n^{-4/5})$. For the integral over $A_1^n$ we simply observe that by \eqref{eq:tail_control}, $\bar{w}_n \leq \beta + n^{-2/5} \log n$, and recall that the measure of $A_1^n$ is at most $2 n^{-4/5}(\beta + \frac{4}{5}\log n)$. Since $b_n-a_n \geq \frac{16}{5\alpha}\log n$, then $\int_{A_1^n} \bar{f}_n$ is also $O((b_n-a_n) n^{-4/5})$.
This implies that the second integral in \eqref{eq:hell_decom} is $O((b_n-a_n) n^{-4/5})$.

Using \eqref{eq:int_Bm}, \eqref{eq:l_infty_bd_intermediate}, Lemma 3.1 of \cite{vandervaart2008} and the above, the third term of \eqref{eq:hell_decom} is bounded by a multiple of
\begin{align*}
& \int_{B_n} e^{w_1} \Bigg( 1-\frac{1}{\sqrt{\int_{B_n}e^{w_1}}} \Bigg)^2  + \int_{B_n} \Bigg( \frac{e^{w_1/2}}{\sqrt{\int_{B_n}e^{w_1}}} - \frac{e^{\bar{w}_n/2}}{\sqrt{\int_{B_n}e^{\bar{w}_n}}} \Bigg)^2  \\
& \quad \quad + \int_{B_n} e^{\bar{w}_n} \Bigg( \frac{1}{\sqrt{\int_{B_n}e^{\bar{w}_n}}} -\frac{1}{\sqrt{\int_{a_n}^{b_n} e^{\bar{w}_n}}} \Bigg)^2 \\
& \lesssim \left( \int_{B_n} f_1 - 1 \right)^2 + \|\bar{w}_n - w_1\|_{L^\infty(B_n)}^2 e^{\|\bar{w}_n - w_1\|_{L^\infty(B_n)}} + \left( \int_{B_n} e^{\bar{w}_n} - \int_{a_n}^{b_n}e^{\bar{w}_n} \right)^2\\
& = O((\log n)^2 n^{-4/5} + (b_n-a_n)^2 n^{-8/5}),
\end{align*}
which establishes $(i)$.

Consider $(iv)$. Note that this is trivial if $f_0(x) = 0$, so assume $f_0(x)\neq 0$. If $x\in B_n$, then by \eqref{eq:l_infty_bd_intermediate},
\begin{equation*}
f_0(x)/\bar{f}_n(x) = e^{w_1(x)-\bar{w}_n(x)} \int f_0 \psi_n \int_{a_n}^{b_n} e^{\bar{w}_n} = e^{O(n^{-2/5}\log n)}(1+o(1)) = 1+o(1).
\end{equation*}
If $x\in [a_n,b_n]\backslash B_n$, then the result follows from \eqref{eq:tail_control}.

Consider lastly $(v)$. Since $\bar{w}_n$ defined in \eqref{eq: concave} is piecewise linear with $|w_+'(a_n)|\vee |w_-'(b_n)|\leq n^{4/5}$, in view of Lemma \ref{lem: mix_rep_lem} it takes the form
\begin{align*}
\bar{w}_n(x)=\gamma_1\sum_{i=1}^M\frac{z_i\wedge (x-a_n)}{z_i}-\gamma_2(x-a_n)+\gamma_3,\quad x\in[a_n,b_n],
\end{align*}
with $M=O(n^{1/5}\log n)$, $\gamma_1\leq |w_+'(a_n)-w_-'(b_n)|(b_n-a_n)\leq 2n^{4/5}(b_n-a_n)$ and $|\gamma_2|\leq |w_-'(b_n)|\leq n^{4/5}$. This completes the proof.
\end{proof}

\subsection{Proof of Lemma \ref{lem:mle_exp_ineq}}\label{sec:tech_remain}

\begin{proof}[Proof of Lemma \ref{lem:mle_exp_ineq}]
It is shown in the proof of Theorem 5 of Kim and Samworth \cite{kim2016}, p. 2772, that for $\eta = \eta(\eps) \in (0,1)$ to be defined below and $L \geq L_0(\eta) = L_0(\eps)$,
\begin{equation}\label{eq:dir}
\sup_{g_0 \in \overline{\F}^{0,1}} P_{g_0}^n \Big( \big\{ h(\hat{g}_n,g_0) \geq Ln^{-2/5} \big\} \cap \big\{ \hat{g}_n \in \widetilde{\F}^{1,\eta} \big\}\Big) \leq 2^{15/2} \exp \left( - \frac{L^2 n^{1/5}}{2^{28}} \right),
\end{equation}
so that it remains only to control $\sup_{g_0 \in \overline{\F}^{0,1}} P_{g_0}^n (\hat{g}_n \not\in \widetilde{\F}^{1,\eta} )$. Lemma 6 of Kim and Samworth \cite{kim2016} shows that this quantity is $O(n^{-1})$ as $n\to\infty$; we essentially follow their proof, suitably sharpening the probability bounds in the case $d=1$. Then
\begin{equation}\label{eq:con_ineq_terms}
\begin{split}
\sup_{g_0 \in \overline{\F}^{0,1}} P_{g_0}^n (\hat{g}_n \not\in \widetilde{\F}^{1,\eta} ) & \leq \sup_{g_0 \in \overline{\F}^{0,1}} P_{g_0}^n (|\mu_{\hat{g}_n}| >1 ) + \sup_{g_0 \in \overline{\F}^{0,1}} P_{g_0}^n (\sigma_{\hat{g}_n}^2 > 1+\eta ) \\
& \quad \quad + \sup_{g_0 \in \overline{\F}^{0,1}} P_{g_0}^n (\sigma_{\hat{g}_n}^2 < 1-\eta ).
\end{split}
\end{equation}
By Lemma 13 of \cite{fresen2013}, there exist universal constants $\alpha_0, \beta_0>0$ such that for all $x\in \R$,
\begin{align*}\label{eq:gba}
\sup_{g\in \overline{\F}^{0,1}} g(x) \leq e^{\beta_0 - \alpha_0|x|}.
\end{align*}
It is shown in \cite{kim2016}, p. 2774, that the first term in \eqref{eq:con_ineq_terms} is bounded by $2\alpha_0^{-1} e^{\beta_0 - \alpha_0 \sqrt{n}}$. For the second term, by Remark 2.3 of D\"umbgen et al. \cite{dumbgen2011}, one has $\sigma_{\hat{g}_n}^2 \leq \widetilde{\sigma}_n^2 := n^{-1} \sum_{i=1}^n (Z_i - \bar{Z}_n)^2 \leq n^{-1} \sum_{i=1}^n Z_i^2$, where $\widetilde{\sigma}_n^2$ denotes the sample variance and $\bar{Z}_n = n^{-1} \sum_{i=1}^n Z_i$ the sample mean. Letting $C_0 = C_0(\alpha_0,\beta_0,2)$ denote the constant in Lemma \ref{lem:exponential_inequality} below, we can apply that lemma to bound the second term in \eqref{eq:con_ineq_terms} by
\begin{equation*}
\sup_{g_0 \in \overline{\F}^{0,1}} P_{g_0}^n \left( \frac{1}{n} \sum_{i=1}^n Z_i^2  >1+\eta  \right) \leq \exp ( -\sqrt{\eta/C_0} n^{1/4} )
\end{equation*} 
for $n \geq \max ( 16C_0^2 / \eta, e^2)$.

Consider now the third term in \eqref{eq:con_ineq_terms}. For $\eps>0$, let $\overline{\cal{P}}_\varepsilon^{1/10,1/2}$ denote the class of probability distributions on $\R$ such that $\mu_P = \int xdP(x)$ and $\sigma_P^2 = \int (x-\mu_P)^2 dP(x)$ satisfy $|\mu_P| \leq 1/10$ and $1/2 \leq \sigma_P^2 \leq 3/2$ and
\begin{equation*}
\int |x|^{2+\varepsilon} dP(x) \leq 4e^{\beta_0} \alpha_0^{-3-\varepsilon} \Gamma (3+\varepsilon) =: \tau_\varepsilon,
\end{equation*}
where $\Gamma(t)=\int_0^\infty x^{t-1}e^{-x}dx$. This is exactly the same as the class $\mathcal{P}^{1/10,1/2}$ considered in Lemma 6 of \cite{kim2016}, except that we have replaced the $4^\text{th}$-moment condition with a $(2+\varepsilon)$-moment condition. Following the rest of the proof of Lemma 6 of \cite{kim2016} (noting that $\sup_{g_0 \in \overline{\F}^{0,1}} \int |x|^{2+\varepsilon} g_0(x) dx \leq \tau_\varepsilon/2$ and that uniform integrability of $\{Y_{n_k}^2: k\in \mathbb{N}\}$ in that proof follows from the $(2+\varepsilon)$-moment condition), one can similarly conclude that for some $\eta =\eta\big(\alpha_0,\beta_0,\varepsilon\big) \in (0,1)$,
\begin{equation*}
\sup_{g_0 \in \overline{\F}^{0,1}} P_{g_0}^n (\sigma_{\hat{g}_n}^2 < 1-\eta ) \leq \sup_{g_0 \in \overline{\F}^{0,1}} P_{g_0}^n \left( \P_n \not\in \overline{\mathcal{P}}_\varepsilon^{1/10,1/2} \right),
\end{equation*}
where $\P_n = n^{-1} \sum_{i=1}^n \delta_{Z_i}$ is the empirical measure and $\delta_x$ is the Dirac measure at $x$. This last probability can be bounded by
\begin{align*}
& \sup_{g_0 \in \overline{\F}^{0,1}} P_{g_0}^n (|\bar{Z}_n|>1/10)  + \sup_{g_0 \in \overline{\F}^{0,1}} P_{g_0}^n (|\widetilde{\sigma}_n^2 -1| > 1/2) \\
& \quad \quad + \sup_{g_0 \in \overline{\F}^{0,1}} P_{g_0}^n \left(\int |x|^{2+\varepsilon} (d\P_n(x) - g_0(x)dx) > \tau_\varepsilon/2 \right).
\end{align*}
Using similar arguments to those used previously, the first two terms can be bounded by $Ce^{-c\sqrt{n}}$ and $Ce^{-cn^{1/4}}$, respectively, where the constants depend only on $\alpha_0$ and $\beta_0$. For the last term, using Lemma \ref{lem:exponential_inequality} with $C_0 = C_0(\alpha_0,\beta_0,2+\varepsilon)$ the constant in that lemma,
\begin{equation*}
\sup_{g_0 \in \overline{\F}^{0,1}} P_{g_0}^n\left( \frac{1}{n} \sum_{i=1}^n \left( |Z_i|^{2+\varepsilon} - \E_{g_0} |Z_i|^{2+\varepsilon} \right) > \frac{\tau_\varepsilon}{2} \right) \leq \exp \left(- \left(\frac{\tau_\varepsilon}{2C_0}\right)^{1/(2+\varepsilon)} n^\frac{1}{4+2\varepsilon}\right)
\end{equation*}
for $n \geq \max (4(2+\varepsilon)^{4+2\varepsilon}C_0^2/\tau_\varepsilon^2,e^{2+\varepsilon})$. In conclusion, we have shown that for any $\eps>0$ there exists $\eta=\eta(\eps)\in(0,1)$ such that
\begin{equation*}
\sup_{g_0 \in \overline{\F}^{0,1}} P_{g_0}^n \big(\hat{g}_n \not\in \widetilde{\F}^{1,\eta} \big) \leq C(\eps) \exp \Big(-c(\varepsilon)n^\frac{1}{4+2\varepsilon}\Big)
\end{equation*}
for all $n\geq n_0(\varepsilon)$, where $C(\eps)$, $c(\eps)$ and $n_0(\eps)$ are positive constants. Together with \eqref{eq:dir}, this establishes the result.
\end{proof}

\begin{lemma}\label{lem:exponential_inequality}
Let $Z_1,...,Z_n$ be i.i.d. random variables from a density $f\in \F_{\alpha,\beta}$, where $\alpha>0$ and $\beta\in \R$. Then for any $r \geq 1$, $t \geq r^r$ and $n\geq e^r$,
\begin{equation*}
P_f^n \left( \frac{1}{\sqrt{n}} \left| \sum_{i=1}^n ( |Z_i|^r - \E|Z_i|^r ) \right| \geq C_0(\alpha,\beta,r) t \right) \leq \exp ( -t^{1/r} ).
\end{equation*}
\end{lemma}

\begin{proof}
For notational convenience, write $\lambda = 1/r \in (0,1]$. Let $x_\lambda = (1/\lambda)^{1/\lambda}$ and define
$$\Psi_\lambda (x)= 
\begin{cases} 
      e^{x^\lambda} & x \geq x_\lambda,\\
      \tau_\lambda x & x <x_\lambda,
   \end{cases}
   $$
where $\tau_\lambda = \Psi_\lambda (x_\lambda) /x_\lambda = (\lambda e)^{1/\lambda}$. This defines a Young function, that is a convex, increasing function $\Psi_\lambda : \R^+ \rightarrow \R^+$ with $\Psi_\lambda(0)=0$. Denote the corresponding Orlicz norm $\|X\|_{\Psi_\lambda} := \inf \{ a>0: \E [\Psi_\lambda(|X|/a)]\leq 1 \}$. Note that the density function $g_\lambda$ of $|Z_1|^{1/\lambda}$ satisfies $g_\lambda(x) = \lambda x^{-(1-\lambda)}\big(f(x^\lambda)+f(-x^\lambda)\big) \1_{\{ x\geq 0\}}\leq \lambda x^{-(1-\lambda)} e^{\beta - \alpha x^\lambda} \1_{\{ x\geq 0\}}$. Then for fixed $a>\alpha^{-1/\lambda}$,
\begin{align*}
\E \Psi_{\lambda} (|Z_1|^{1/\lambda}/a) & \leq \frac{\lambda \tau_\lambda e^\beta}{a} \int_0^{x_\lambda} u^\lambda e^{-\alpha u^\lambda} du + \lambda e^\beta \int_{x_\lambda}^\infty u^{-(1-\lambda)} e^{-(\alpha-a^{-\lambda})u^\lambda} du \\
& = K_0 (a,\alpha,\beta,\lambda) < \infty.
\end{align*}
If $K_0 = K_0(a,\alpha,\beta,\lambda)> 1$, it follows by convexity that $\| |Z_1|^{1/\lambda} \|_{\Psi_{\lambda}} \leq aK_0$.

By Theorem 6.21 of Ledoux and Talagrand \cite{ledoux2011}, there exists a constant $K_\lambda$ such that
\begin{equation*}
\left\| \sum_{i=1}^n ( |Z_i|^\frac{1}{\lambda} - \E |Z_i|^\frac{1}{\lambda}) \right\|_{\Psi_\lambda} \leq K_\lambda \left( \left\| \sum_{i=1}^n (|Z_i|^\frac{1}{\lambda} - \E |Z_i|^\frac{1}{\lambda}) \right\|_1 + \left\| \max_{1\leq i \leq n} \big| |Z_i|^\frac{1}{\lambda} - \E |Z_i|^\frac{1}{\lambda}\big| \right\|_{\Psi_\lambda} \right).
\end{equation*}
The first-term on the right-hand side can be bounded by the $\|\cdot \|_2$-norm of the same quantity, which equals the square root of $n\E(|Z_i|^\frac{1}{\lambda} - \E |Z_i|^\frac{1}{\lambda})^2 \leq n \E|Z_i|^\frac{2}{\lambda}$. For any $\delta \in (0,1]$,
\begin{equation*}
\E |Z_1|^\frac{1}{\delta} = \int_0^\infty x g_{\delta}(x)dx \leq 2\delta e^\beta \int_0^\infty x^\delta e^{-\alpha x^\delta} dx = \frac{2e^\beta}{\alpha^{1+\frac{1}{\delta}}} \int_0^\infty y^\frac{1}{\delta} e^{-y} dy = \frac{2e^\beta}{\alpha^{1+\frac{1}{\delta}}} \Gamma ( 1 + 1/\delta).
\end{equation*}
Since this is finite for any $\delta \in (0,1]$, we can bound the $\|\cdot\|_1$-norm above by $C(\alpha,\beta,\lambda) \sqrt{n}$.

Note that for any random variable $X$, $\|X-\E X\|_{\Psi_\lambda} \leq 2 \|X\|_{\Psi_\lambda}$. Indeed, setting $a = \|X\|_{\Psi_{\lambda}}$, since $\Psi_\lambda$ is convex and increasing,
\begin{equation*}
\E \Psi_\lambda \left( \frac{|X-\E X|}{2a} \right) \leq \frac{1}{2} \E \Psi_\lambda \left( \frac{|X|}{a} \right) + \frac{1}{2} \Psi_\lambda \left( \frac{\E| X|}{a} \right) \leq \E \Psi_\lambda \left( \frac{|X|}{a} \right) \leq 1.
\end{equation*}
Using this and Lemma 2.2.2 of van der Vaart and Wellner \cite{aadwellbook},
\begin{equation*}
\big\| \max_{1\leq i \leq n} \big| |Z_i|^\frac{1}{\lambda} - \E |Z_i|^\frac{1}{\lambda} \big| \big\|_{\Psi_\lambda} \leq K(\Psi_\lambda) \Psi_\lambda^{-1} (n) \max_{1\leq i \leq n} \big\| |Z_i|^\frac{1}{\lambda} \big\|_{\Psi_\lambda} \leq K(\alpha,\beta,\lambda)(\log n)^\frac{1}{\lambda},
\end{equation*}
so that we have shown
\begin{equation*}
\left\| \sum_{i=1}^n (|Z_i|^\frac{1}{\lambda} - \E |Z_i|^\frac{1}{\lambda}) \right\|_{\Psi_\lambda} \leq C(\alpha,\beta,\lambda) \sqrt{n}.
\end{equation*}
By Markov's inequality, for any random variable $X$, $\P( |X| \geq x\|X\|_{\Psi_\lambda}) = \P (\Psi_\lambda (|X|/\|X\|_{\Psi_\lambda}) \geq \Psi_\lambda (x) ) \leq 1/\Psi_\lambda(x)$. Applying this to the above sum completes the proof.
\end{proof}

\section{Additional simulations}\label{sec:extra_sim}

We firstly provide some additional figures for the empirical Bayes posterior of the mode. We consider the $Beta(2,5)$ and $Gamma(2,1)$ distributions and take i.i.d. samples of size ranging from $n=50$ to $n=20000$ in both cases. As in the Gaussian example, we run the Gibbs sampler for $20000$ iterations of which half were considered as burn in and discarded. The resulting marginal posteriors for the mode based on the empirical Bayes posterior are displayed in Figures  \ref{fig: mod_Beta} and \ref{fig: mod_gamma}.

\begin{figure}[h!]
\begin{center}
\includegraphics[width = 0.9\textwidth]{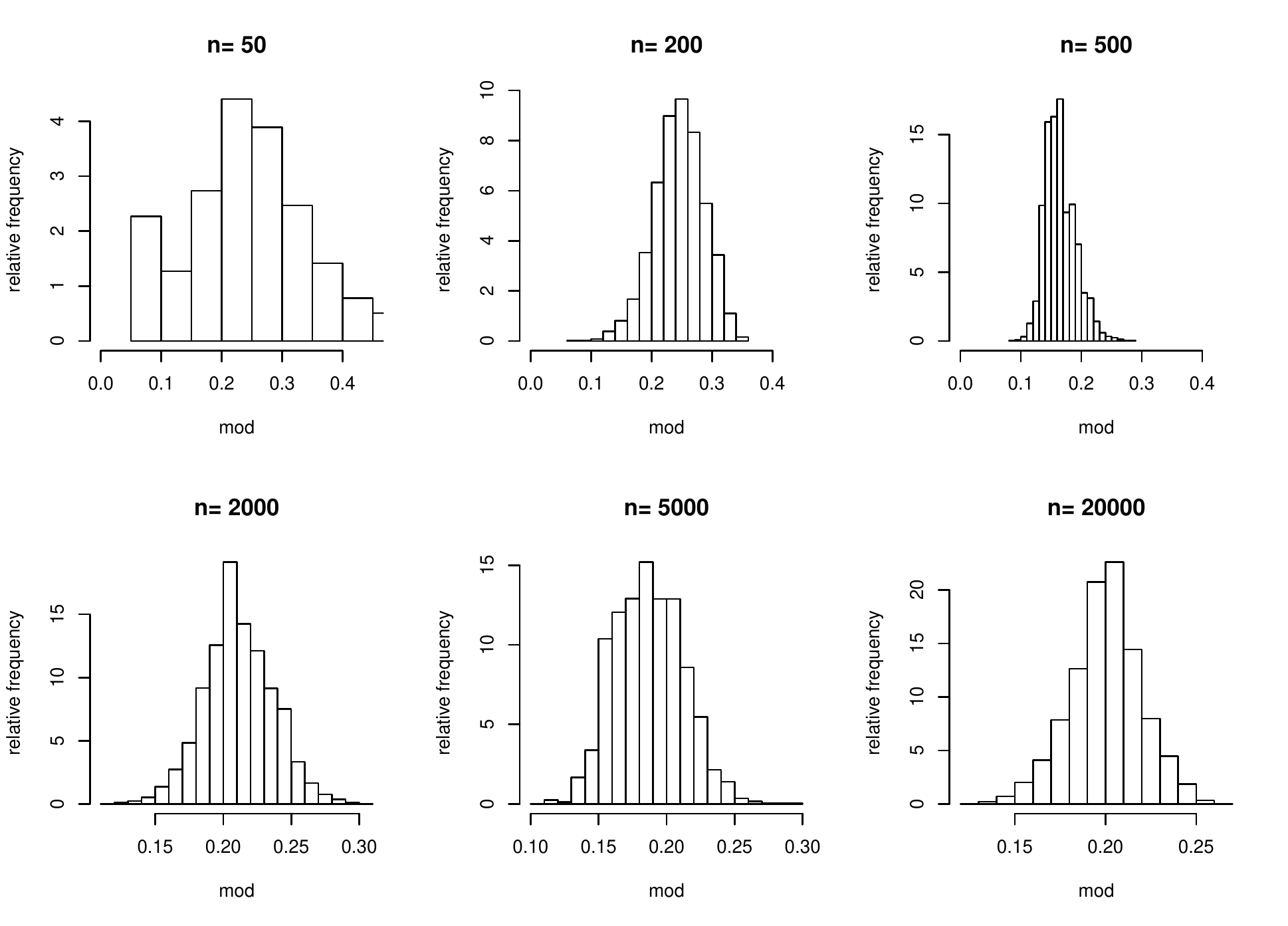} 
\caption{The marginal posterior distribution for the mode based on the empirical Bayes posterior for the $Beta(2,5)$ distribution with increasing sample size from left to right and top to bottom, ranging from $n=50$ to $n=20000$.
}
\label{fig: mod_Beta}
\end{center}
\end{figure}

\begin{figure}[h!]
\begin{center}
\includegraphics[width = 0.9\textwidth]{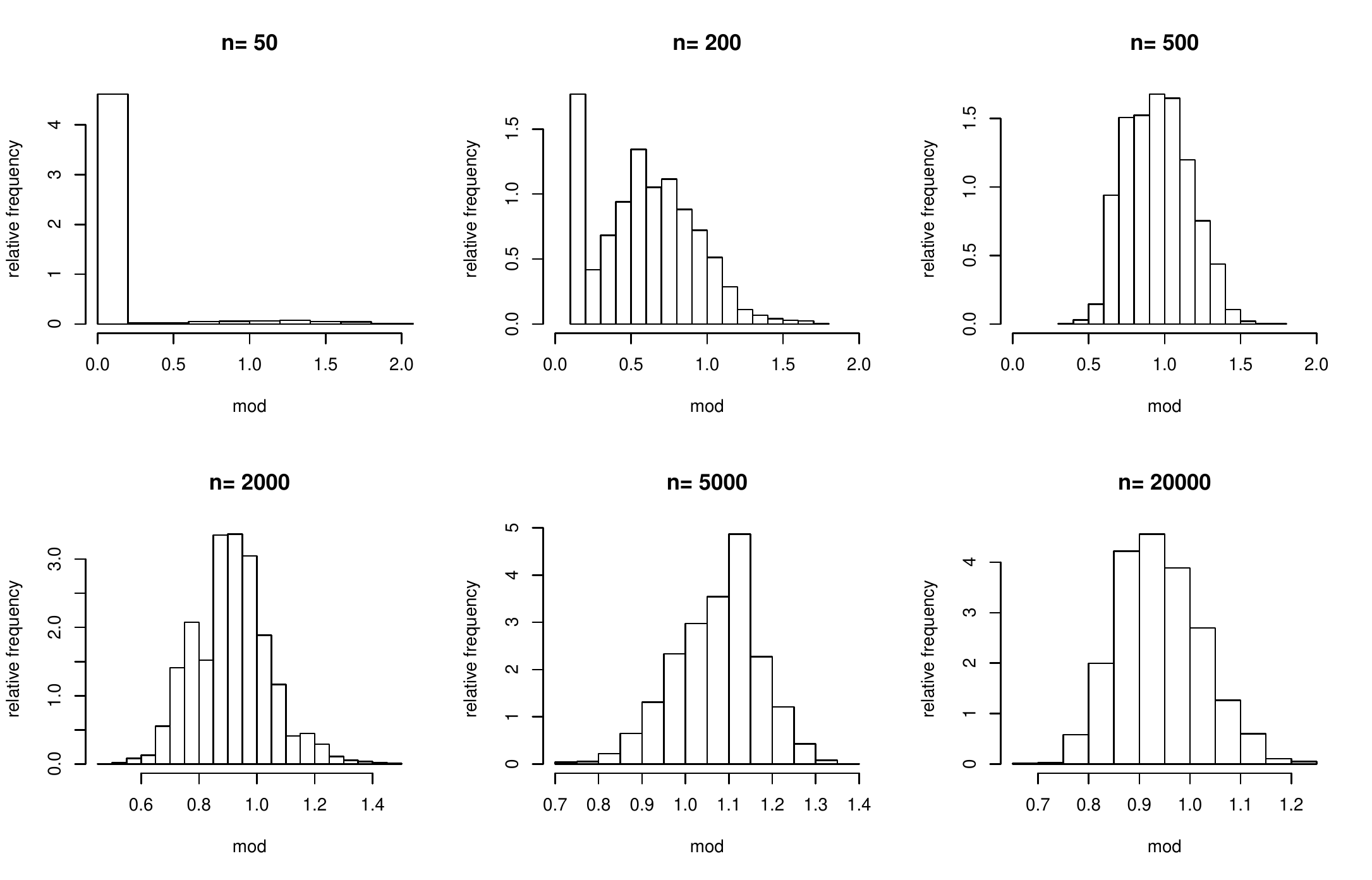} 
\caption{The marginal posterior distribution for the mode based on the empirical Bayes posterior for the $Gamma(2,1)$ distribution with increasing sample sizes from left to right and top to bottom, ranging from $n=50$ to $n=20000$.
}
\label{fig: mod_gamma}
\end{center}
\end{figure}

Next we provide a brief simulation study to demonstrate the (possible) applicability of our procedure for cluster analysis. Here we consider only the simple case where the underlying density is a mixture of two log-concave densities. We have modified our Bayesian procedure to accommodate mixtures of log-concave densities in a straightforward way, again using a Gibbs sampler. We consider various combinations of log-concave densities including two Gaussians, two Laplace distributions, a Laplace and a Gamma distribution and two Beta distributions. We consider two sample sizes, $n=100$ and $n=500$, see Figures \ref{fig: mixture100} and   \ref{fig: mixture500}, respectively. In all pictures we have plotted the posterior mean (solid blue), $95\%$ pointwise credible sets (dashed blue) and the underlying density (solid red). One can see that in all cases our procedure provides reasonable estimators and seemingly reliable uncertainty quantification.\\

\begin{figure}
\subfloat{\includegraphics[width = 3.1in]{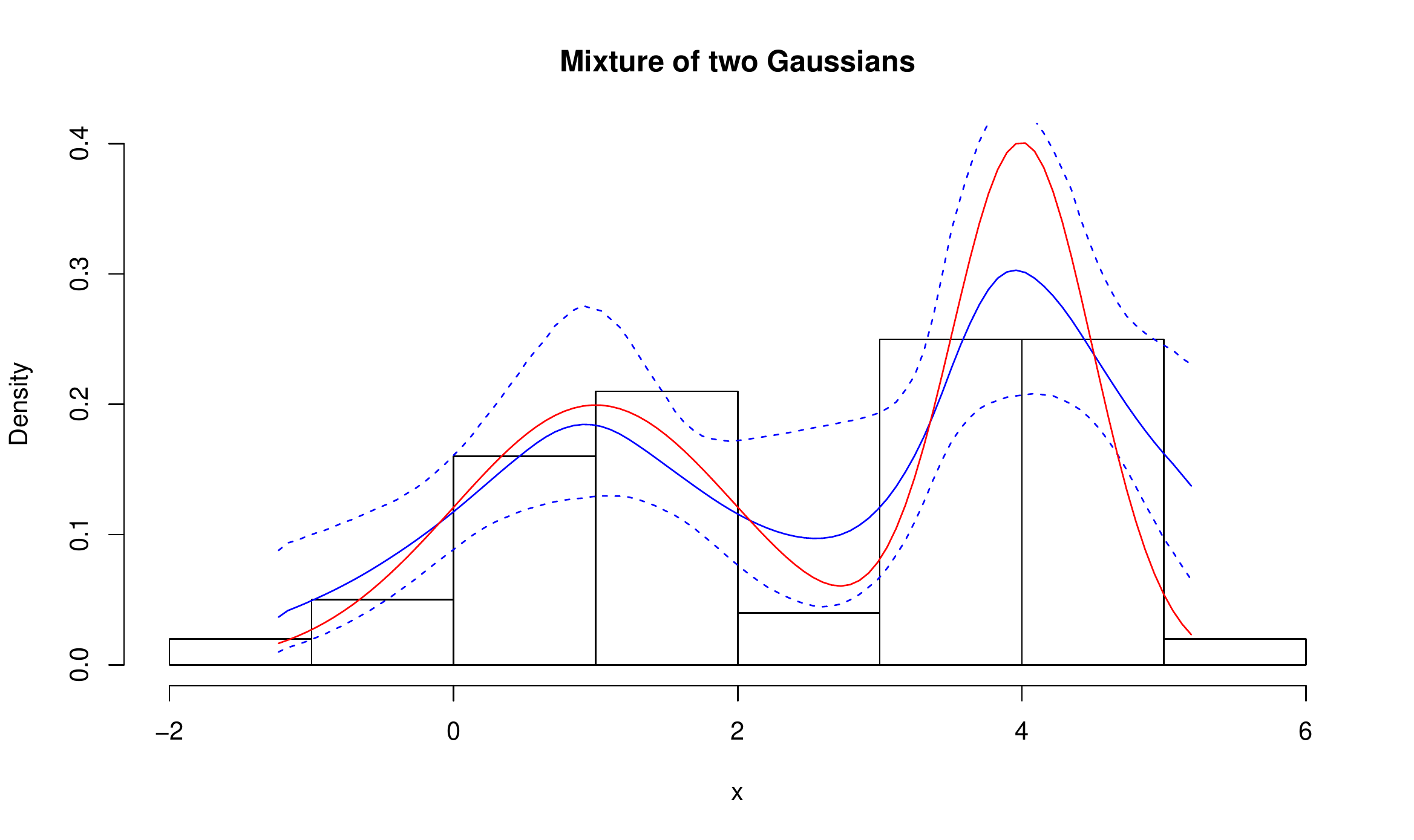}} 
\subfloat{\includegraphics[width = 3.1in]{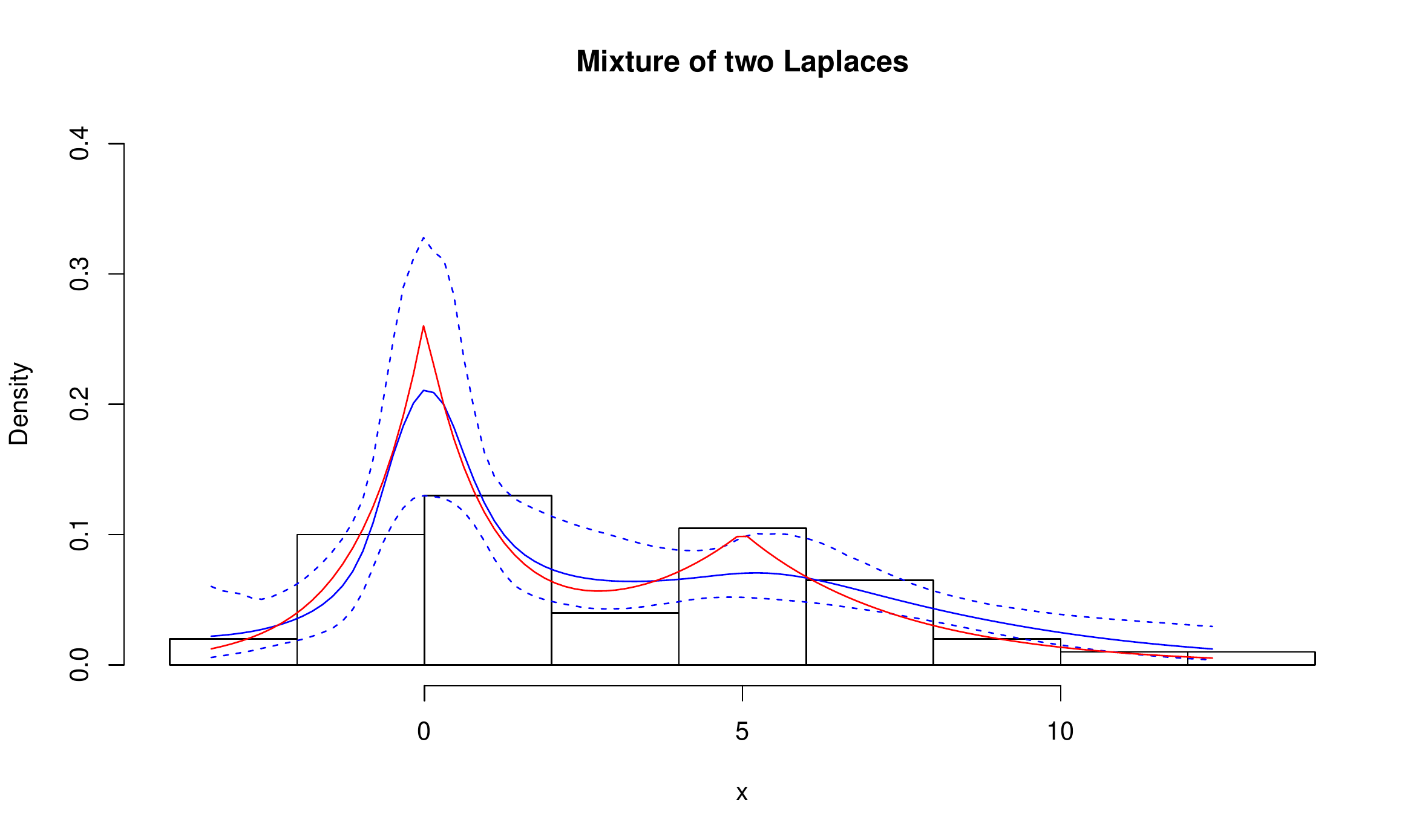}}\\
\subfloat{\includegraphics[width = 3.1in]{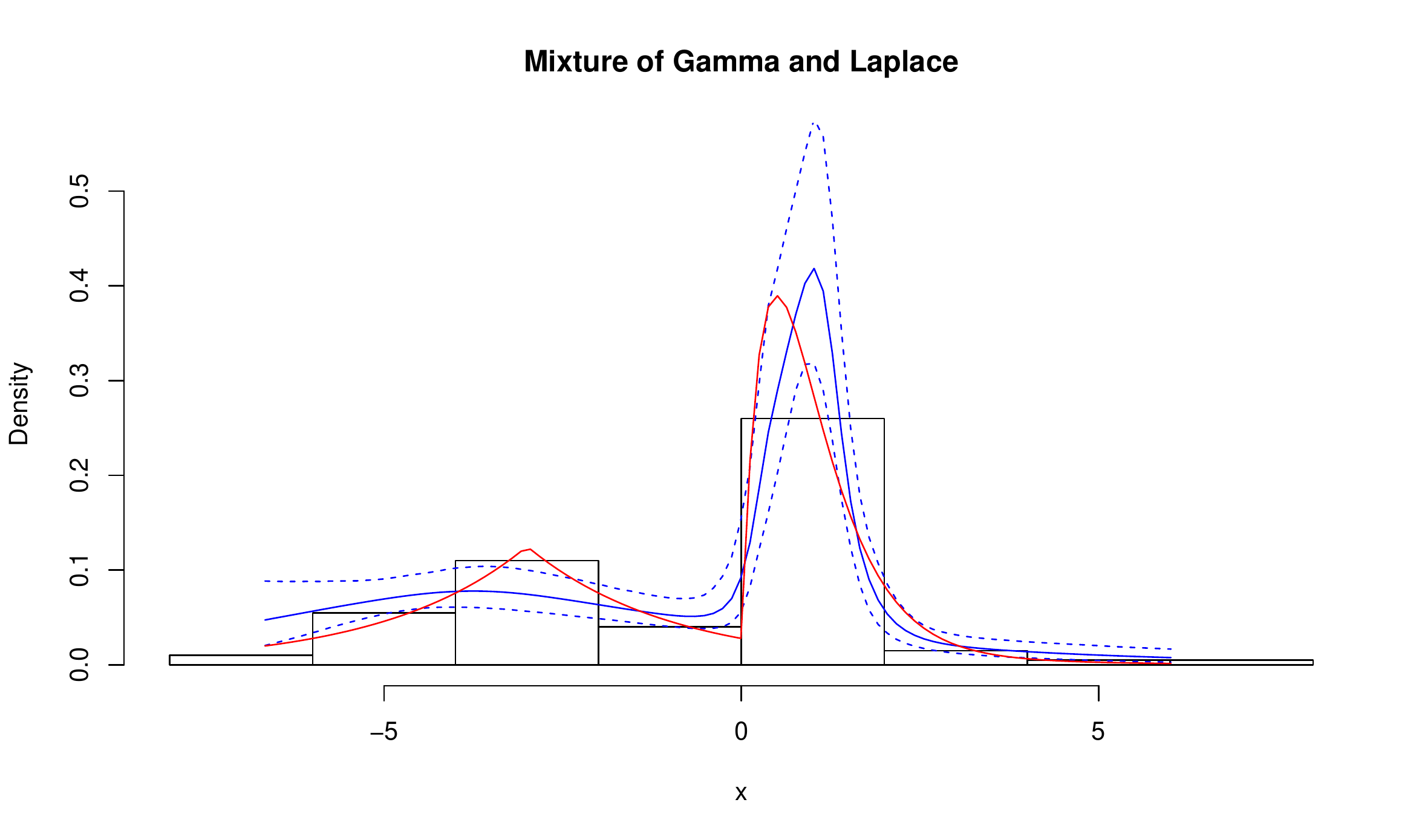}}
\subfloat{\includegraphics[width = 3.1in]{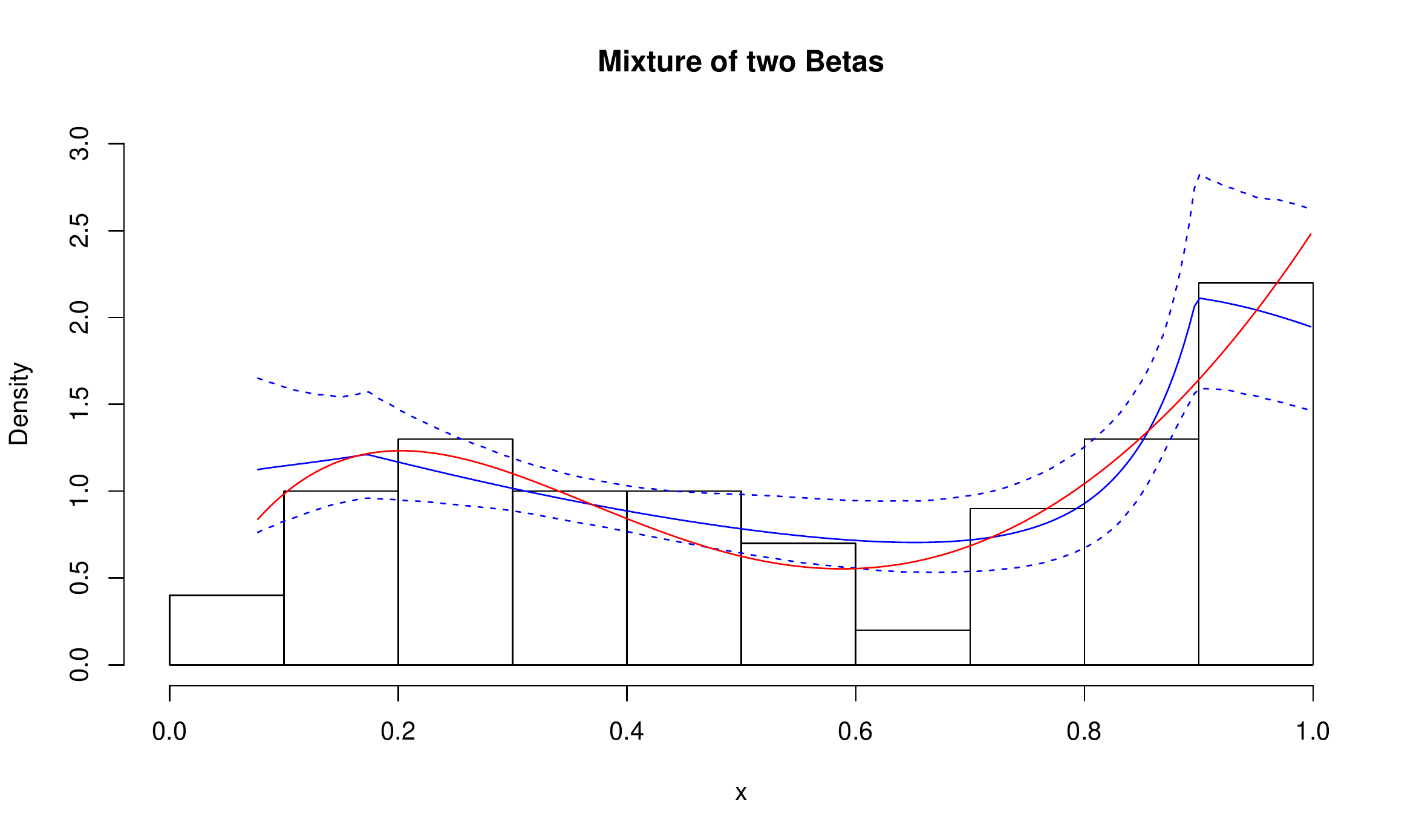}} 
\caption{Empirical Bayes posterior distribution for the mixture of two log-concave densities for sample size $n=100$. Top left corner: $0.5 N(1,1)+0.5N(4,0.5)$, top right corner: $0.5*Laplace(0,1)+0.5*Laplace(5,2.5)$, bottom left corner: $0.5*Gamma(2,2)+0.5* Laplace(-3,2)$, bottom right corner: $0.5*Beta(2,5)+0.5*Beta(5,1)$.}
\label{fig: mixture100}
\end{figure}

\begin{figure}
\subfloat{\includegraphics[width = 3.1in]{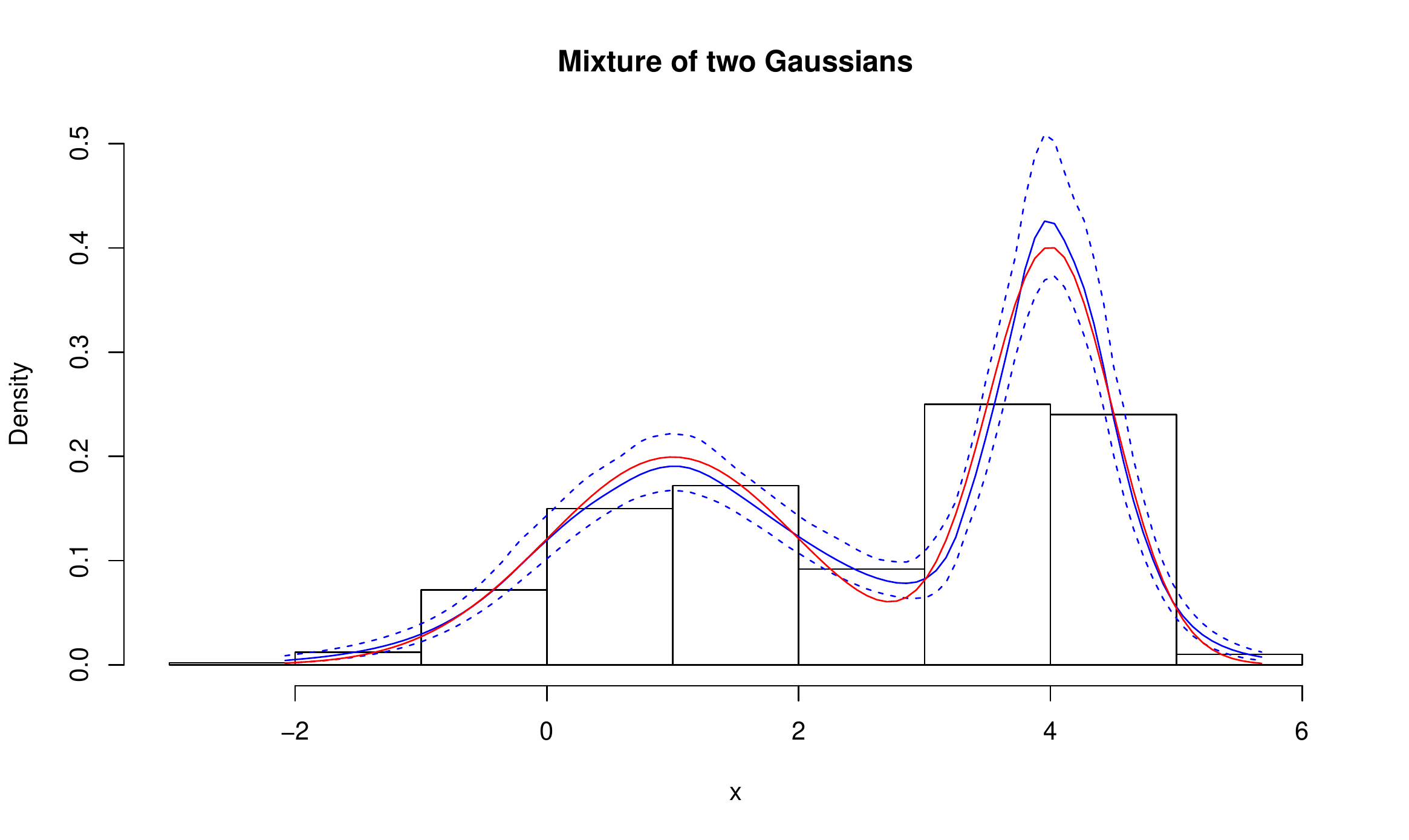}} 
\subfloat{\includegraphics[width = 3.1in]{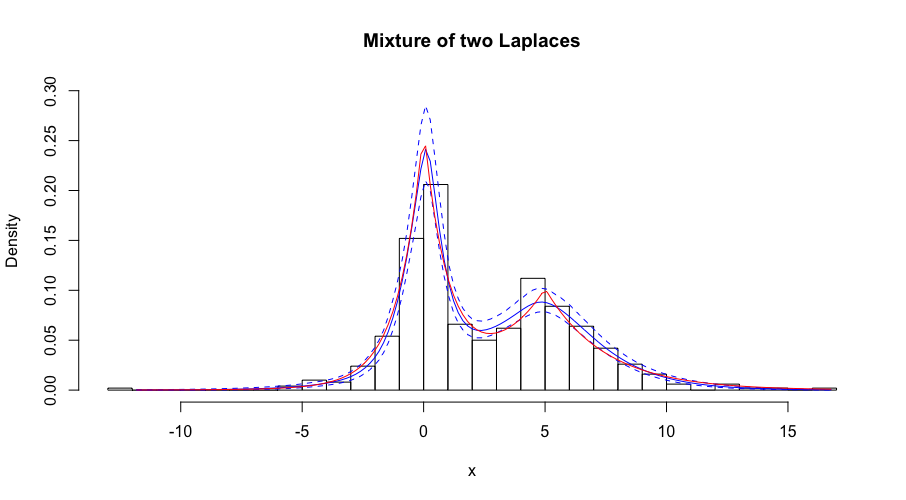}}\\
\subfloat{\includegraphics[width = 3.1in]{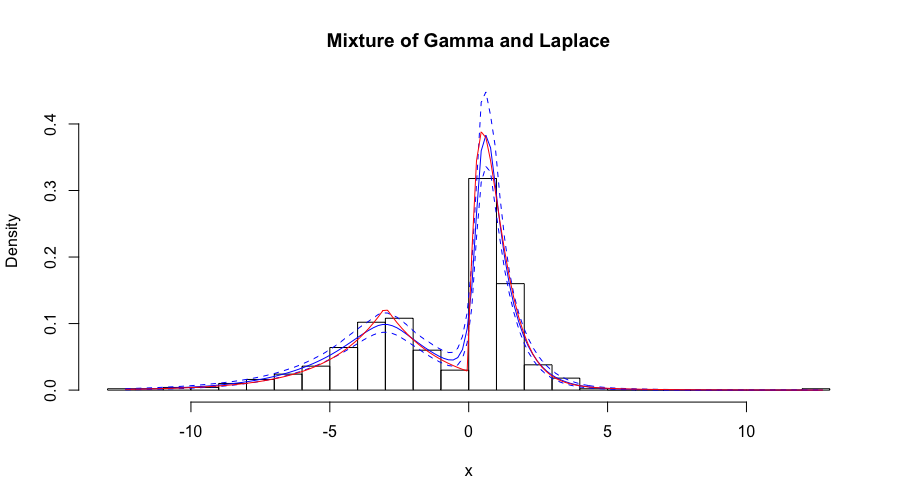}}
\subfloat{\includegraphics[width = 3.1in]{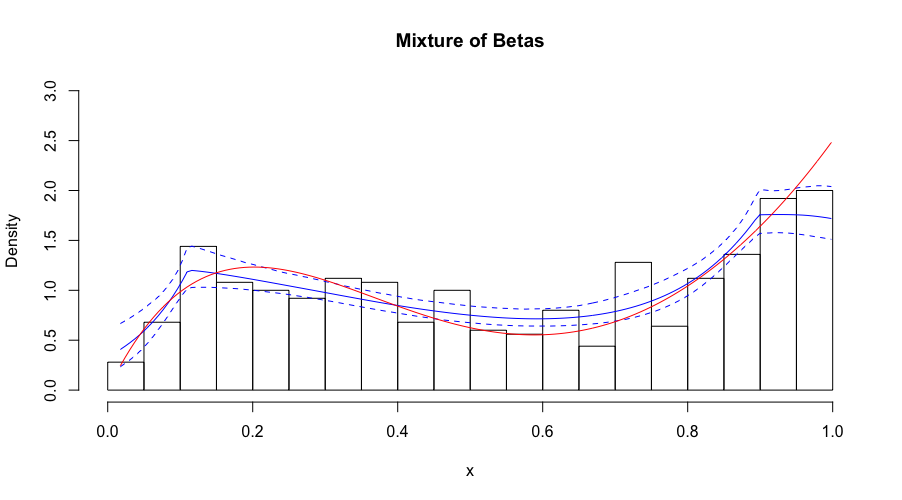}} 
\caption{Empirical Bayes posterior distribution for the mixture of two log-concave densities for sample size $n=500$. Top left corner: $0.5 N(1,1)+0.5N(4,0.5)$, top right corner: $0.5*Laplace(0,1)+0.5*Laplace(5,2.5)$, bottom left corner: $0.5*Gamma(2,2)+0.5* Laplace(-3,2)$, bottom right corner: $0.5*Beta(2,5)+0.5*Beta(5,1)$.}
\label{fig: mixture500}
\end{figure}

\textbf{Acknowledgements:} The authors would like to thank Richard Samworth and Arlene Kim for helpful discussions. The authors would further like to thank the AE and two referees for their helpful suggestions which lead to an improved version of the manuscript.

All three authors received funding from the European Research Council under ERC Grant Agreement 320637 for this research. Ester Mariucci was further supported by the Federal Ministry for Education and Research through the Sponsorship provided by the Alexander von Humboldt Foundation, by the Deutsche Forschungsgemeinschaft (DFG, German Research Foundation) – 314838170,
GRK 2297 MathCoRe, and by Deutsche Forschungsgemeinschaft (DFG) through grant CRC 1294 'Data Assimilation'. Botond Szab\'o also received funding from the Netherlands Organization for Scientific Research (NWO) under Project number: 639.031.654.

\bibliography{refs}{}
\bibliographystyle{acm}

\end{document}